\documentclass[11pt,reqno]{amsart}
\usepackage{amsmath,amsthm,amssymb,amsfonts,enumerate,color,enumerate} 
\newcommand{\mysection}[1]{\section{#1}

      \setcounter{equation}{0}}

\newcommand\cbrk{\text{$]$\kern-.15em$]$}} 
\newcommand\opar{\text{\raise.2ex\hbox{${\scriptstyle | }$}\kern-.34em$($} }

\usepackage{color}

\DeclareMathOperator*{\esssup}{ess\,sup}

\newtheorem{theorem}{Theorem}[section]
\newtheorem{lemma}[theorem]{Lemma}
\newtheorem{proposition}[theorem]{Proposition}
\newtheorem{corollary}[theorem]{Corollary}

\theoremstyle{definition}
\newtheorem{assumption}{Assumption}[section]
\newtheorem{definition}{Definition}[section]

\theoremstyle{remark}
\newtheorem{remark}{Remark}[section]

\newcommand\bH{\mathbb{H}}

\newcommand\bL{\mathbb{L}}
\newcommand\bR{\mathbb{R}}
\newcommand\bP{\mathbb{P}}
\newcommand\bB{\mathbb{B}}

\newcommand\bZ{\mathbb{Z}}

\newcommand\frc{\mathfrak{c}}
\newcommand\frf{\mathfrak{f}}

\newcommand\frh{\mathfrak{h}}

\newcommand\fru{\mathfrak{u}}

\newcommand\frF{\mathfrak{F}}

\newcommand\frH{\mathfrak{H}}
\newcommand\frG{\mathfrak{G}}
\newcommand\frm{\mathfrak{m}}
\newcommand\frn{\mathfrak{n}}

\newcommand\cB{\mathcal{B}}
\newcommand\cD{\mathcal{D}}

\newcommand\cF{\mathcal{F}}

\newcommand\cH{\mathcal{H}}

\newcommand\cO{\mathcal{O}}
\newcommand\cR{\mathcal{R}}

\newcommand\cV{\mathcal{V}}
\newcommand\cZ{\mathcal{Z}}
\newcommand\cW{\mathcal{W}}

\makeatletter
 \newcommand{\sumstar}%
 {\operatornamewithlimits{\sum@\kern-.2em\raise1ex\hbox{*}}}
 \makeatother

\newcommand{\F}{\mathcal{F}}
\newcommand{\R}{\mathbb{R}}
\newcommand{\E}{\mathbb{E}}
\renewcommand{\P}{\mathbb{P}}

\def\Xint#1{\mathchoice
{\XXint\displaystyle\textstyle{#1}}%
{\XXint\textstyle\scriptstyle{#1}}%
{\XXint\scriptstyle\scriptscriptstyle{#1}}%
{\XXint\scriptscriptstyle\scriptscriptstyle{#1}}%
\!\int}
\def\XXint#1#2#3{{\setbox0=\hbox{$#1{#2#3}{\int}$ }
\vcenter{\hbox{$#2#3$ }}\kern-.6\wd0}}

\def\dashint{\Xint-}

\usepackage[UKenglish]{babel}
\usepackage{mathabx}
\usepackage{stmaryrd}
\usepackage{latexsym, amssymb}
\usepackage{amsopn}
\usepackage{mathrsfs}
\usepackage{esint}

\begin{document}
 
\title[On Stochastic Navier-Stokes Equations] 
{On uniqueness of solutions to stochastic Navier-Stokes Equations}

\author[R. Cotter]{Raymond Cotter}%
\thanks{}
\address{School of Mathematics and Maxwell Institute,
University of Edinburgh,
King's  Buildings,
Edinburgh, EH9 3JZ, United Kingdom}
\email{cotterraymond30@gmail.com}

\author[I. Gy\"ongy]{Istv\'an Gy\"ongy}
\address{School of Mathematics and Maxwell Institute,
University of Edinburgh,
King's  Buildings,
Edinburgh, EH9 3JZ, United Kingdom}
\email{i.gyongy@ed.ac.uk}

\subjclass[2020] {35Q30, 35R60, 60H15}
\keywords{ Navier-Stokes equations, stochastic Navier-Stokes equations, 
weak solutions, conditional 
uniqueness, Ladyzhenskaya-Prodi-Serrin condition, Morrey spaces}

\begin{abstract}
Theorems on uniqueness and continuous dependence on the initial condition 
of solutions to stochastic Navier-Stokes equations 
driven by Wiener processes and Poisson martingale measures are presented. 
These theorems generalise some results from \cite{GK2026}.   
\end{abstract}

\dedicatory{Dedicated to Professor Bohdan Maslowski on the occasion of his 70th birthday}
\maketitle

\mysection{Introduction}
The Navier--Stokes equations
\begin{equation}
                                                    \label{eq1}
\partial_t u=\nu\Delta u-(u\cdot\nabla)u-\nabla p+D_j\frf^j+f,
\quad
\operatorname{div}u=0,
\end{equation}
posed on a domain $\cR\subseteq\mathbb R^d$ with boundary and 
initial conditions 
\begin{equation} 
                                                                    \label{bin1}
u|_{\partial \cR}=0, \quad u|_{t=0}=u_0,
\end{equation} 
describe the motion of incompressible viscous fluids (when $d=2,3)$).  
Here $u_t(x)=(u^1_t(x),\dots,u^d_t(x))$ 
and $p_t(x)$ are the velocity of the fluid and the pressure at time $t\geq0$ 
at $x\in\cR$, and $\nu>0$ is the viscosity
coefficient. 

Classical results of Leray \cite{L1934} and Hopf \cite{H1951} show that for
every given initial condition $u_0\in\mathcal H$, force fields 
$\frf^j\in L_2([0,\infty],L_2(\cR,\bR^d))$ $(j=1,2,...,d)$ and 
$f\in L_1([0,\infty),L_2(\cR,\bR^d))$ there exists at least one weak
solution $u$ from  the Hopf--Leray class
$$
\mathcal W:=L_{\infty}([0,\infty),\mathcal H)
\cap L_2([0,\infty),\mathcal V),
$$
where
$\mathcal V$ and $\cH$ denote the closure of the space 
$$
\cD=\{\varphi\in C^{\infty}_0(\cR,\bR^d): \,\text{div$\,\varphi=0$}\}
$$
in $W^1_2=W^1_2(\cR,\bR^d)$ and $L_2=L_2(\cR,\bR^d)$, respectively. 
Here and throughout the paper  $C^{\infty}_0(\cR,\bR^d)$ denotes the space of $\R^d$-valued 
smooth functions on $\cR$ with compact support. 

By a theorem of Ladyzhenskaya \cite{L1958} in 1958 
and by Lions and Prodi \cite{LionsProdi1959} in 1959,  
the weak solution in this class is unique when
$d=2$,  whereas the corresponding problem in dimensions 
$d\ge3$ remained open until a recent paper \cite{ABC} by  
Albritton, Bru\'e and Colombo,  
 where the authors show the existence of more than one solution 
to the 3D Navier-Stokes equations with an appropriate force field 
and zero initial condition.  In \cite{HouWangYang2026} 
the authors demonstrate by a computer-assisted 
proof the non-uniqueness of a weak solution 
for the unforced  3D Navier-Stokes equations. 

These developments further motivate the search 
for additional assumptions ensuring uniqueness.
There is an extensive literature 
on such results, often called conditional uniqueness theorems. 
The best-known conditional uniqueness criterion 
is the Ladyzhenskaya-Prodi-Serrin condition, which assumes
\begin{equation}
                                                                                       \label{conditionLPS}
u\in L_q([0,T], L_p(\cR)),
\quad \frac{d}{p}+\frac{2}{q}\le1, \quad p\in(d,\infty),\quad q\in[2,\infty). 
\end{equation}
Under this condition, uniqueness of Hopf--Leray weak solutions was established
by Prodi \cite{Prodi1959} and Serrin \cite{Serrin1963}, uniqueness and regularity
results were obtained by Ladyzhenskaya \cite{L1967}, 
 and Serrin \cite{Serrin1963}. 

In 
\cite{GK2026} the authors generalised the conditional uniqueness 
theorems in Prodi \cite{Prodi1959} and Serrin \cite{Serrin1963}, and extended them 
to stochastic Navier--Stokes equations driven by 
Wiener processes when $\cR=\bR^d$ 
for $d\geq3$. 

Motivated by \cite{GK2026}, our  purpose in this paper 
is to extend these conditional uniqueness results in three directions:

\begin{itemize}
\item to stochastic Navier–Stokes equations driven simultaneously 
by Wiener processes and compensated Poisson random measures;

\item from the whole space to arbitrary Lipschitz domains;

\item from classical Ladyzhenskaya–Prodi–Serrin assumptions to a substantially 
weaker Morrey-type condition.
\end{itemize}
To present a special case of our main result applied to deterministic Navier-Stokes,  
we introduce the following restricted Hopf-Leray class. 

For any $T>0$ and $r\in(2,d]$ ($d\geq3)$ a function $u:[0,T]\times\cR\to\bR^d$ 
from $\cW$ is said to be in the restricted Hopf-Leray class if $u$ can be represented 
as $u^{M}+u^{B}$, such that for a constant $\hat u=\hat u(d,r)$ we have 
$$
\left(\frac{1}{|B|}\int_{B}{\bf1}_{\cR}|u^M_t(x)|^{r}
\,dx\right)^{1/r}\leq \hat u\rho_{B}^{-1} \quad\text{for $dt$-a.e. $t\in[0,T]$}, 
$$
for any ball $B$ of radius $\rho_{B}\leq 1$, and for $\bar u_t=\esssup_{x\in\cR}|u^B_t(x)|$ 
we have 
$$
\int_0^T|\bar u_t|^2\,dt<\infty. 
$$ 

Our main result, Theorem \ref{theorem3} below, 
for the deterministic problem  \eqref{eq1}-\eqref{bin1}
reads as follows. 
\begin{theorem}
                      \label{theorem1}
There exists a constant $N$ depending only on $d$ and $r$ such that 
if $u$ is a solution to \eqref{eq1}-\eqref{bin1} from the restricted Leray-Hopf 
class satisfying 
$$
\hat u\leq \nu/N
$$
then this is the only solution to  \eqref{eq1}-\eqref{bin1} from 
the restricted Leray-Hopf class. 
 \end{theorem}
We further show (Corollary \ref{corollary1}) 
that the assumption of this theorem is strictly weaker than the classical 
Ladyzhenskaya–Prodi–Serrin condition.

Assume now that $\cR$ is a bounded Lipschitz domain and note that 
by the Poincar\'e inequality 
\begin{equation}
                             \label{Poincare}
|v|_{L_2}\leq \frc|Dv|_{L_2}
\end{equation}
holds for $v\in H^1_0$ with a constant $\frc=\frc(\cR)$, where $H^1_0$ denotes 
the closure of $C^{\infty}_{0}(\cR,\bR^d)$ in $W^1_2(\cR,\bR^d)$. 

The next theorem illustrates our stability result, Theorem \ref{theorem4},
which generalises a result in Serrin \cite{Serrin1963}. 

\begin{theorem}
                             \label{theorem stability}
 There exists a constant $N>0$ depending only on $d$, $r$ and $\frc$ 
 such that if $u$ is a solution to \eqref{eq1}-\eqref{bin1} from the restricted 
 Leray-Hopf class, satisfying 
\begin{equation*}
                                                 \label{stability}
 \tilde u:=\hat u+\esssup_{t\in[0,T]}\bar u_t<\nu/N, 
\end{equation*}
 then we have 
 $$
 |u_t-v_t|_{\cH}\leq |u_0-v_0|_{\cH}e^{-\gamma t}\quad\text{for $t\in[0,T]$}
 $$
 for any solution $v$ to \eqref{eq1}-\eqref{bin1} from the restricted Leray-Hopf 
 class with initial condition $v_t|_{t=0}=v_0\in\cH$, where 
 $\gamma:=\frc^{-2}(\nu-N\tilde u)>0$.  
 \end{theorem}

 In this paper we study 
 stochastic Navier–Stokes equations driven simultaneously 
 by multiplicative Wiener noise and compensated Poisson random measures, 
 \begin{align}
du=&\big(\nu\Delta u-(u\cdot\nabla) u
+\nabla p+D_j\frf_j +f \big)\,dt                                                                  \nonumber\\
&+\big(\sigma\nabla u-\nabla q+h\big)\circ dw
+\int_{Z}\frh_t(u_t,z)\tilde \pi(dz,dt), \,\,\text{div}\, u=0                                                       \label{NS}
\end{align}
for a random velocity field $u=(u_t^1(x),...,u^d_t(x))$,  
and random pressure fields $p=p_t(x)$ 
and $q=q_t(x)$ for $t\in[0,T]$ and $x\in\cR$, with boundary and initial conditions
\eqref{bin1}
for a random velocity field $u=(u_t^1(x),...,u^d_t(x))$,  
and random pressure fields $p=p_t(x)$ 
where $w$ is a Hilbert space-valued 
Wiener process, $\circ dw_t$ indicates that the corresponding differential 
is understood in the Stratonovich sense, and $\tilde\pi$ is a compensated 
Poisson random measure. 
Such equations arise naturally from stochastic Lagrangian particle dynamics 
and extend the models considered in  \cite{MR2001} and 
\cite{MR2004} and \cite{MR2005} by incorporating jump noise.

When $\cR=\bR^d$ and $\frh=0$, the authors in \cite{MR2005} 
studied the Cauchy problem \eqref{NS}-\eqref{bin1}
with a second order elliptic differential operator 
instead of $\Delta$,  and with force fields 
$D_j\frf_j$, $f$ and $h$, which may depend also on $u$. 
They proved the existence of a solution on some filtered probability 
space $(\Omega,\cF,(\cF_t)_{t\geq0},P)$ 
carrying an $\cF_t$-Wiener process $w$ with values in a Hilbert space. 
Moreover, in $d=2$ they show the existence 
and uniqueness of a unique strong solution.

The problem of existence and uniqueness of solutions of stochastic 
Navier--Stokes equations with random noises with jumps has been 
studied by many authors, see, e.g., \cite {ZBL2019} 
and the references therein. 

We consider for $d\geq3$ the type of stochastic Navier-Stokes equations 
as in \cite{MR2005}, but with the additional last term in \eqref{NS}, see 
equations \eqref{SNS5}--\eqref{div5}--\eqref{bin5}. 
Using the method developed in \cite{KG1980} for rewriting stochastic 
It\^o integrals with respect to martingale measures as stochastic It\^o 
integrals with respect to Hilbert space-valued martingales, we reformulate equations 
\eqref{SNS5}--\eqref{div5}--\eqref{bin5} as the system \eqref{SNS2}--\eqref{bin2}.
We prove our main results, Theorems \ref{theorem2} and \ref{theorem3}, 
for the system \eqref{SNS2}--\eqref{bin2}. As a consequence, 
we obtain the corresponding results for the system 
\eqref{SNS5}--\eqref{div5}--\eqref{bin5},  see Theorem \ref{theorem5}.

Theorem \ref{theorem2} states, in particular, that admissible solutions 
of \eqref{SNS2}--\eqref{bin2} (see Definitions \ref{definition admissible} 
and \ref{definition AS}) admit (strongly) 
c\`adl\`ag $\cH$-valued modifications on the interval $[0,T]$. 
Theorem \ref{theorem3} establishes continuous dependence on the initial 
data and, as a consequence, uniqueness within the class of admissible solutions. 
In the deterministic setting, it immediately yields Theorem \ref{theorem1}.

Furthermore, Theorem \ref{theorem3} yields Corollary \ref{corollary1}, 
which asserts that if $u$ is a solution of the stochastic 
Navier--Stokes equations \eqref{SNS2}--\eqref{bin2} satisfying 
$u\in L_{p,q}$ almost surely for a pair $(p,q)$ fulfilling 
the Ladyzhenskaya--Prodi--Serrin condition \eqref{conditionLPS}, 
then uniqueness holds within the class of admissible solutions. 
This corollary generalises  the classical uniqueness results of Ladyzhenskaya 
\cite{L1967}, Prodi \cite{Prodi1959}, and Serrin \cite{Serrin1963}, 
and extends it to the stochastic Navier--Stokes equations.

In a subsequent paper we present regularity results for admissible solutions 
solutions to  equations \eqref{SNS2}-\eqref{bin2}.

Stochastic Navier--Stokes equations have been studied extensively 
over the past five decades. The pioneering work \cite{BT1973} 
initiated the mathematical analysis of these equations, 
and was followed by a substantial body of work on the existence 
of solutions under various assumptions and in different functional settings; 
see, for example, \cite{AC1990,CC1991,FY1992,BCF1992,CG1994,B1995,DD2002}.

The two-dimensional case is now comparatively well understood. 
In particular, existence and uniqueness results have been established in a variety of frameworks; 
see, for example, \cite{AF2002,MS2002,AF2004,F2003,MR2005,M2009}. 
More recently, stochastic Navier--Stokes equations driven by jump noise 
have also received considerable attention. Local well-posedness 
in $L_p$-spaces for equations with L\'evy noise was obtained in \cite{MS2017}, 
while existence and uniqueness in an  $L_p$-framework for the two-dimensional equations 
with jump noise were proved in \cite{ZBL2019}. 
For the three-dimensional equations, local well-posedness in critical  
$L_p$-spaces with multiplicative cylindrical noise on the 
whole space was established in \cite{KWX2024}.

Several related directions have also been investigated. 
Global existence and uniqueness of mild solutions to the 
three-dimensional random vorticity equations for sufficiently 
small initial vorticity were proved in \cite{BR2017}. 
Ergodic properties of stochastic Navier--Stokes equations 
have been studied extensively, with results in two dimensions obtained in \cite{KS2000,KS2001,WMS2001,WM2001,B2002,MH2006} 
and in three dimensions in \cite{DD2003}. Balance relations 
for randomly forced Navier--Stokes equations on the two-dimensional 
torus were established in \cite{KP2005}. More recently, \cite{AV2024} 
proved well-posedness, regularization, 
and blow-up criteria for stochastic Navier--Stokes equations 
on the $d$-dimensional torus ($d\geq2$), together with a global well-posedness 
result with high probability for sufficiently small initial data in critical spaces.

For a comprehensive account of stochastic fluid dynamics, 
we refer the reader to the monograph \cite{F2008}. 

While the well-posedness theory for stochastic 
Navier--Stokes equations is now well developed in many settings, 
considerably less is known about conditional uniqueness results in three dimensions. 
The present paper contributes to this direction 
by establishing conditional uniqueness for equations 
driven simultaneously by Wiener processes and compensated 
Poisson martingale measures on general Lipschitz domains.

The paper is organised as follows. Section 2 introduces t
he stochastic Navier--Stokes equations driven 
by a `cylindrical martingale', the notion of admissible solutions, 
and the relevant function spaces. It also states our main results, 
Theorems \ref{theorem2} and \ref{theorem3}, 
together with Theorem~\ref{theorem4}, 
which demonstrates how the approach developed 
in the paper can be applied to establish exponential stability of admissible solutions 
to the deterministic Navier--Stokes equations. 
Section 3 presents a ``stochastic energy equality", 
i.e., an It\^o formula for the square of the norm 
of semimartingales with values in a Gelfand triple of Hilbert spaces, 
which plays an important role 
in the proof of Theorems \ref{theorem2}, \ref{theorem3} 
and \ref{theorem4}. 
Section 4 is dedicated to the proof of Theorems \ref{theorem2}, 
\ref{theorem3} and \ref{theorem4}. 
Finally, Section~5 applies Theorems~\ref{theorem2} and \ref{theorem3} 
to stochastic Navier--Stokes equations driven simultaneously 
by Wiener processes and compensated Poisson random measures.

We conclude with introducing some notations. 
We use the notation 
$$
(f,g)=\int_{\cR}f^i(x)g^i(x)\,dx|, \quad |f|_{L_2}=\sqrt{(f,f)} 
$$
for vector fields $f=(f^i)_{i=1}^d$ and $g=(g^i)_{i=1}^d$ on $\cR$, when 
the integrals exist.  
Here and throughout the paper, unless otherwise indicated, 
the summation convention with respect to repeated integer-valued 
indices is in force.  

 Let
$S: L_2(\cR, \bR^d)\to\cH$
denote the projection of $L_2$ onto $\mathcal{H}$, and define
$R:=I-S$, 
where $I$ denotes the identity operator on 
$L_2$. For scalar functions $\varphi$ and vector fields
$u=(u^1,\dots,u^d)$
on $\cR$, we use the notations 
$$
D_i\varphi=\frac{\partial \varphi}{\partial x_i},
\quad
\nabla \varphi=D\varphi=(D_1\varphi,\dots,D_d\varphi), 
\quad \varphi_{(u)}:=(u\cdot\nabla)\varphi=u^jD_j\varphi,
$$
and we write $Du=(D_j u^i)_{i,j=1}^d$ for  the Jacobian matrix of $u$. 

For balls $B$ in $\bR^{d}$ and Lebesgue integrable 
functions $f$ on $B$ 
we use the notation 
$$
\dashint_{B}f\,dx:=\frac{1}{|B|}\int_{B}f\,dx, 
$$        
where $|B|$ denotes the volume of $B$. 
For a Banach space $\bB$, we denote by
$\ell_2(B)$
the space of sequences
$c=(c_n)_{n\in\bZ_0}, \, c_n\in B$ equipped with the norm
$$
|c|_{\ell_2(\bB)}=\left(
\sum_{n\in \bZ}^\infty |c_n|_{\bB}^2\right)^{1/2}, 
$$ 
where $\bZ$ is the set of integers. 
In particular, we write $\ell_2:=\ell_2(\mathbb{R})$. 
For topological spaces $S$ we write $\mathcal{B}(S)$  
for the Borel $\sigma$-algebra on $S$.

Throughout the paper all random variables 
and stochastic processes are defined on a filtered probability space
$(\Omega,\mathcal{F},(\mathcal{F}_t)_{t\ge0},\mathbb{P})$
carrying a sequence
$(m^k)_{k\in\bZ}$
of  c\`adl\`ag square integrable $\mathcal{F}_t$-martingales with 
predictable covariation processes  $<m^i,m^j>_t=t\delta_{i,j}$ for $i,j\in\bZ$, 
a sequence $(w^k)_{k=1}^{\infty}$ of independent 
$\cF_t$-Wiener processes and an $\cF_t$-Poisson martingale measure 
$\tilde\pi(dz,dt)=\pi(dz,dt)-\frn(dz)\otimes dt$, where $\pi(dz,dt)$ is an 
$\cF_t$-Poisson random measure with a $\sigma$-finite characteristic measure 
$\frn(dz)$ 
on a measurable space $(Z,\cZ)$ with a countably generated $\sigma$-algebra $\cZ$.   
If $h=(h_t^k)_{k\in\bZ}$ is an $\ell_2(\bR^d)$-valued optional process such that 
$$
\int_0^T|h_t|^2_{\ell_2(\bR^d)}\,dt<\infty \,\,\text{(a.s.)}, 
$$
then one can the define the stochastic integral 
$$
\int_0^th^k_s\,dm_s^k
=\lim_{n\to\infty}\sum_{k=1}^{n}\int_0^th^k_s\,dm^k_s
+\lim_{n\to\infty}\sum_{-n\leq k\leq 0}\int_0^th^k_s\,dm^k_s, 
$$
where the limits exist in probability,  uniformly in $t\in[0,T]$. 
We assume that $(\cF_t)_{t\geq0}$ is right-continuous 
and $\F_0$ contains all subsets of $\P$-null sets of $\F$. 
We denote by $\cO$ optional 
$\sigma$-algebra on $\Omega\times[0,\infty)$.

\mysection{ Formulation of the Main Theorems}

Let $\cR\subseteq\bR^d$ be Lipschitz domain and $T\in(0,\infty)$.  
We consider the stochastic Navier-Stokes equations                                               
\begin{align}
du_t
&=\Big(D_i\big(a_t^{ij}D_j u_t + \frf_t^i(u_t)\big)
+ f_t(u_t,Du_t)- u_{t(u_t)}-\nabla p_t+ \gamma_t^{ki}D_i q_t^k\Big)\,dt \nonumber\\
&\quad+\Big(\sigma_t^{ik}D_i u_t+ h_t^k(u_t)-\nabla q_t^k\Big)\,dm_t^k,\label{SNS2}
\quad
\operatorname{div}u_t=0
\end{align}
on $\Omega\times[0,T]\times\cR$, 
with boundary and initial 
conditions 
\begin{equation}
                                                         \label{bin2}
u_t|_{\partial \cR}=0, \quad \text{for $t\in[0,T]$\quad 
and
\quad 
$u_t(x)|_{t=0}=u_0(x)$,
\quad$ x\in \cR$},
\end{equation}
for a velocity field
$
u=(u^1_t(x),\dots,u^d_t(x)),
$
and pressure fields
$
p=p_t(x)$ and $
q=(q_t^k(x))_{k\ge1}
$
on $\Omega\times[0,T]\times\cR$. 
We assume that

(i)
the coefficient
$a=(a^{ij})$
is $\mathbb{R}^{d\times d}$-valued,
$\gamma^i=(\gamma^{ik})_{k\in\bZ}$
is an $\ell_2(\mathbb{R}^d)$-valued function, and
$\sigma^i=(\sigma^{ik})_{k\in\bZ}$
is an $\ell_2$-valued 
$\cO\otimes\mathcal B(\mathbb \cR)$-measurable function on
$\Omega\times[0,T]\times\mathbb \cR$ for  $i=1,\dots,d$;

(ii)
the functions
$\frf^i$ and $f$ are $\bR^d$-valued,
$\frf^i$ is a $\cO\otimes\mathcal B(\cR\times\bR^d)$ -measurable 
mapping on $\Omega\times[0,T]\times\cR\times\bR^d$, 
and $f$ is a $\cO\otimes\cB(\cR\times\bR^d\times\bR^{d\times d})$
-measurable mapping on 
$\Omega\times[0,T]\times\cR\times\bR^d\times\bR^{d\times d}$,
for each $i=1,\dots,d$. The function $h=(h^k)_{k\in\bZ}$ is an $\ell_2(\mathbb R^d)$-valued
$\cO\otimes\mathcal B(\cR\times\mathbb R^d)$-measurable mapping on
$\Omega\times[0,T]\times\cR\times\bR^d$. 

\begin{definition}
                                                            \label{hl_def}
A $\mathcal V$-valued optional process
$u=(u_t)_{t\in[0,T]}$
is a solution to \eqref{SNS2}--\eqref{bin2} on $[0,T]$ if
$u\in L_2([0,T],\mathcal V)\cap L_\infty([0,T],\mathcal H)\, \text{(a.s.)}, $
and for each
$\varphi\in\mathcal V$
the identity
\begin{equation} 
                                             \label{hl_sol}
\begin{aligned}
(u_t,\varphi)&=(u_0,\varphi)-
\int_0^t\Big[(a_s^{ij}D_j u_s+\mathfrak{f}_s^i(u_s),D_i \varphi)
+(u_{s(u_s)},\varphi)\Big]\,ds                              \\
&\quad
+\int_0^t\big(\gamma_s^{ki}D_i q_s^k+f_s(u_s,Du_s),\varphi\big)\,ds\\
&\quad
+\int_0^t\big(\sigma_s^{ik}D_i u_s+h_s^k(u_s),\varphi\big)\,dm_s^k
\end{aligned}
\end{equation}
holds for $\mathbb P\otimes dt$-almost every
$(\omega,t)\in\Omega\times[0,T]$, 
and 
\begin{equation}
                                 \label{pressure_assumption}
\nabla q_t^k=R\big(\sigma_t^{ik}D_i u_t+h_t^k(u_t)\big)
\quad
(\bP\otimes dt\otimes dx\text{-a.e.})\quad \text{for $k\in\bZ$}. 
\end{equation}
We call
$(u_t)_{t\in[0,T]}$
an $\mathcal H$-solution if it is an $\mathcal H$-valued 
$\mathcal F_t$-adapted c\`adl\`ag process such that
$u\in L_2([0,T],\mathcal V)\ \text{(a.s.)}$, 
and almost surely \eqref{hl_sol} holds for all
$t\in[0,T],\,\varphi\in\mathcal V$, 
and \eqref{pressure_assumption} is satisfied.
\end{definition} 

\begin{remark}
By Assumptions \ref{assumption1} and  \ref{assumption2} below,  
\eqref{pressure_assumption} implies that, for every solution $u$,
$$
M_t^k(u_t)-\nabla q_t^k=\mathcal S M_t^k(u_t)\in \mathcal H,
\quad
\mathbb P\otimes dt\otimes dx\text{-a.e.},
\quad k\in\bZ_0,
$$
where
\begin{equation}
                           \label{stochastic_integrand}
M_t^k(u):=\sigma_t^{ik}D_i u+h_t^k(u),
\quad
u\in\mathcal V.
\end{equation}
Moreover, by \eqref{pressure_assumption},
\begin{equation}
                         \label{closure_assumption}
\gamma_t^{ki}D_i q_t^k
=\gamma_t^{ki}\big(R M_t^k(u)\big)^i,
\quad k\in\bZ_0,
\end{equation}
where
$\big(R M_t^k(u)\big)^i$
denotes the $i$-th component of $\mathcal R M_t^k(u)$. 
Consequently, the variational formulation \eqref{hl_sol} is closed in $u$.
\end{remark}

\begin{assumption}
                                                            \label{assumption1}
\begin{enumerate}
\item[(i)] There exists a constant $\delta>0$ such that
\begin{equation}
\label{coercivity}
|a_t|
\le
\delta^{-1},
\qquad
\Big(
a_t^{ij}
-
\frac12\sigma_t^{ik}\sigma_t^{jk}
\Big)\lambda^i\lambda^j
\ge
\delta|\lambda|^2,
\end{equation}
for all
$
\lambda\in\mathbb R^d,
\, (\omega,t,x)\in\Omega\times[0,T]\times\cR.
$

\item[(ii)]
There exists a nonnegative optional process
$
\vartheta=(\vartheta_t)_{t\in[0,T]}
$
such that almost surely
\begin{equation}
                                                    \label{g_uperbound}
\int_0^T \vartheta_t^2\,dt<\infty\quad \text{ and } \quad |\gamma_t(x)|^2
:=
\sum_{i=1}^d\sum_{k\in\bZ_0}
|\gamma_t^{ki}(x)|^2
\leq \vartheta_t^2 
\end{equation}
 for all
$
(t,x)\in[0,T]\times\cR.
$
\end{enumerate}
\end{assumption}

\begin{remark}
                                           \label{remark sigma}
By \eqref{coercivity},
$|\sigma_t|\leq N(d,\delta)$. 
\end{remark}

\begin{assumption} 
                                                   \label{assumption2}
There exist nonnegative optional processes
$\lambda,\kappa,\chi$, 
and nonnegative $\cO\otimes\mathcal B(\cR)$-measurable functions
$\mathfrak{F}$, $\mathfrak{G}$
on $\Omega\times[0,\infty]\times\cR$ such that
\begin{equation*}
\int_0^T(\lambda_t^2+\chi_t^2+\kappa_t)\,dt<\infty,
\quad
\int_0^T\big(|\mathfrak{F}|_{L_2}^2+|\mathfrak{G}|_{L_2}\big)\,dt<\infty
\quad \text{(a.s.),}
\end{equation*}
and 
\begin{enumerate}
\item[(i)]
for all
$\omega\in\Omega$, $t\geq0$, $x\in\cR$, $u\in\mathbb R^d$, 
$z\in\mathbb R^{d\times d}$,
\begin{equation}
                                            \label{f_upperbound}
|\mathfrak{f}_t(x,u)|+|h_t(x,u)|
\leq
\lambda_t|u|+\mathfrak{F}_t(x),
\end{equation}
$$
|f_t(x,u,z)|\leq \kappa_t|u|+\chi_t|z|+\mathfrak{G}_t(x).
$$
\item[(ii)]
for all $\omega\in\Omega$, $t\geq0$, 
$x\in\cR$, $u_1, u_2\in\bR^d$ and  
$z_1, z_2\in\bR^{d\times d}$,
\begin{equation}
                                                  \label{lipschitz1}
|\mathfrak{f}_t(x,u_1)-\mathfrak{f}_t(x,u_2)|
+
|h_t(x,u_1)-h_t(x,u_2)|
\le
\lambda_t|u_1-u_2|,
\end{equation}
\begin{equation}
                                                   \label{lipschitz2}
|f_t(x,u_1,z_1)-f_t(x,u_2,z_2)|
\le
\kappa_t|u_1-u_2|
+
\chi_t|z_1-z_2|.
\end{equation}
\end{enumerate}
\end{assumption}

\begin{assumption}
                                         \label{assumption3}
The initial condition
$u_0$
is an $\mathcal F_0$-measurable $\mathcal H$-valued random variable.
\end{assumption}

Assume $d\geq3$ and let $r\in(2,d]$ 
be a constant. Let
$\mathbb B_\rho$ 
denote the family of balls in $\mathbb R^d$ of radius $\rho$.

\begin{definition}
                                                                 \label{definition admissible}
A real-, vector-, or tensor-valued function
on $\Omega\times[0,T]\times\mathbb \cR$
is called admissible if
$f=f^M+f^B$  
for some $\cO\otimes\cB(\mathbb\cR)$-measurable 
functions $f^M$, $f^B$, and 
a constant $\hat f>0$ and a nonnegative optional process
$\bar f=(\bar f_t)_{t\in[0,T]}$ exist such that 
\begin{equation}
                                              \label{admisdable_MN_upperbound}
\left(
\dashint_{B_\rho}{\bf1}_{\cR}|f^M(t,x)|^r\,dx
\right)^{1/r}
\le
\hat f\,\rho^{-1} \quad \text{for $\mathbb{P}\otimes dt$-a.e 
$(\omega,t)\in\Omega\times[0,T]$}
\end{equation}
for all $B_\rho\in\mathbb B_\rho$, $\rho\leq1$, 
and
\begin{equation}
                                               \label{admissible_essup_upperbound}
\operatorname*{ess\,sup}_{x\in\cR}
|f_t^B(x)|
\leq \bar f_t \,\, \text{(a.s.)},
\quad
\sup_{\omega\in\Omega}
\int_0^T \bar f_t^2\,dt<\infty.
\end{equation}
\end{definition}
The components $f^M$ and $f^B$ are referred to as {\it Morrey} and 
{\it bounded} components of the admissible function $f$. 
\begin{definition}
For$r\in[1,\infty)$, $\lambda\ge0$, we denote by
$E_{r,\lambda}$ the Morrey space of Euclidean-valued functions $f$ 
on $\bR^d$ satisfying
$$
|f|_{r,\lambda}:=\sup_{\rho\in(0,1],\,B\in\mathbb B_\rho}
\rho^\lambda\left(\dashint_{B} {\bf1}_{\cR}|f(x)|^r\,dx\right)^{1/r}<\infty.
$$
\end{definition}
Thus \eqref{admisdable_MN_upperbound} is equivalent to
$
|f^M|_{r,1}\le \hat f 
$ $( \mathbb{P}\otimes dt \text{-}a.e.)$.

\begin{definition}
                                                                          \label{definition AS} 
An $\mathbb R^d$-valued function $u$ 
on $\Omega\times[0,T]\times\cR$ is called an admissible solution to
\eqref{SNS2}--\eqref{bin2} if $u$ is an admissible function and
$(u_t)_{t\in[0,T]}$
is a solution in the sense of Definition~\ref{hl_def}.
\end{definition}

\begin{theorem}
                                 \label{theorem2}
Let Assumptions \ref{assumption1}, \ref{assumption2}(i) 
and \ref{assumption3} hold, and let $u$ be an admissible solution to
\eqref{SNS2}--\eqref{bin2}. Then $u$ admits a
$\bP\otimes dt$-modification, denoted again by $u$, 
which is an $\mathcal H$-solution. Moreover,                                        
\begin{align}
&\E\sup_{t\le T}|e^{-\varphi_t}u_t|_{\mathcal H}^2
+\E\int_0^T\alpha_t|e^{-\varphi_t}u_t|_{\mathcal H}^2\,dt
+\E\int_0^T|e^{-\varphi_t}u_t|_{\mathcal V}^2\,dt.  \nonumber\\
&\qquad\leq
N \E|u_0|_{\mathcal H}^2
+
N \E\int_0^T
|e^{-\varphi_t}\frF_t|_{ L_2}^2\,dt
+
N \E\left(\int_0^T|e^{-\varphi_t}\frG_t|_{L_2}\,dt\right)^2, \nonumber
\end{align}
where
$N=N(d,\delta)$ is a constant, 
and
$$
\varphi_t
=
\int_0^t \alpha_s\,ds 
\quad
\varepsilon <\alpha_s
=N'\big(\lambda_s^2+\chi_s^2+\vartheta_s^2+\kappa_s\big)
+\mu_s,
$$
with constants
$N'=N'(d,\delta)>0$,  $\varepsilon >0$
and any nonnegative optional process 
$\mu$ satisfying
$$
\int_0^T \mu_t\,dt<\infty
\quad\text{(a.s.).}
$$
\end{theorem}

\begin{theorem}
                                            \label{theorem3}
Let Assumptions \ref{assumption1}, 
\ref{assumption2} and \ref{assumption3} hold.  
Then there exists a constant 
$N=N(d,r)>0$
such that if $u$ is an admissible solution 
to \eqref{SNS2}--\eqref{bin2} satisfying 
$u\leq \delta/N$, then 
$$
\E e^{-\phi_{t}}|u_t-v_t|^2_{\cH}
\leq \E |u_0-v_0|_{\cH}^2\quad\text{for $t\in[0,T]$} 
$$
for any admissible solution $v$ to \eqref{SNS2}--\eqref{bin2} 
on $[0,T]$ with an $\cH$-valued initial condition 
$v_t|_{t=0}=v_0$, where 
$$
\phi_t=N\int_0^t\beta_s\,ds,
\quad 
\beta_s=\hat u^2+\bar u_s^2+\lambda^2_s+\vartheta^2+\chi^2_s+\kappa_s
$$
and $N=N(d,r,\delta)$ is a constant. In particular, if $u_0=v_0$ 
then almost surely $u_t=v_t$ for all $t\in[0,T]$ for the $H$-valued c\`adl\`ag  
version of the solutions. 
\end{theorem}

Finally we present a theorem on exponential stability of admissible 
solutions to deterministic Navier-Stokes equations, i.e., when the functions 
$a$, $\frf$, $f$ and the initial condition in equations \eqref{SNS2}--\eqref{bin2} 
do not depend on $\omega$, and 
$\sigma=0$, $q=0$ and $h=0$. 

This means we consider the equations 
\begin{align}
du_t=&D_i\big(a_t^{ij}D_j u_t + \frf_t^i(u_t)\big)\,dt \nonumber\\
&+ \big(f_t(u_t,Du_t)- u_{t(u_t)}-\nabla p_t\big)\,dt, 
 \quad
\operatorname{div}u_t=0                            \label{NS3}
\end{align}
on $[0,T]\times\cR$, 
with boundary and initial 
conditions 
\begin{equation}
                                                         \label{bin3}
u_t|_{\partial \cR}=0, \quad \text{for $t\in[0,T]$\quad 
and
\quad 
$u_t(x)|_{t=0}=u_0(x)$,
\quad$ x\in\cR$},
\end{equation}
where $a=(a_t^{ij}(x))$ is an $\mathbb{R}^{d\times d}$-valued Borel 
function on $[0,\infty)\times\bR^d$, the functions 
$\frf^i=\frf^i_t(x,u)$ and $f=f_t(x,u,z)$ are $\bR^d$-valued Borel functions on 
$[0,\infty)\times\cR\times\bR^d$ and 
$[0,\infty)\times\cR\times\bR^d\times\bR^{d\times d}$, respectively 
for each $i=1,\dots,d$, and $u_0$ is a Borel function on $\cR$ with values 
in $\bR^d$.  

\smallskip

Let $\cR$ be a bounded Lipschitz domain and 
make the following assumptions. 

\begin{assumption}
                                                            \label{assumption1a}
There exists a constant $\delta>0$ such that
$$
|a_t(x)|\leq\delta^{-1},
\quad
a_t(x)^{ij}\lambda^i\lambda^j\geq\delta|\lambda|^2
$$
for all $(t,x)\in [0,T]\times\cR$ and $\lambda\in\bR^d$.
\end{assumption}

\begin{assumption} 
                                                   \label{assumption2a}
There exist a constant $L$ and Borel functions
$\mathfrak{F}$, $\mathfrak{G}$ on $[0,T]\times\cR$ such that
\begin{equation*}
                                              \label{integrability}
\int_0^T\big(|\mathfrak{F}|_{L_2}^2+|\mathfrak{G}|_{L_2}\big)\,dt<\infty, 
\end{equation*}
\begin{equation*}
|\frf_t(x,u)| \leq L|u|+\frF_t(x),
\quad |f_t(x,u,z)|\leq L(|u|+|z|)+\frG_t(x)
\end{equation*}
for all $t\in[0,T]$, $x\in\cR,\,u\in\mathbb R^d$, 
$z\in\mathbb R^{d\times d}$, and 
\begin{equation}
                                                            \label{lipschitz1a}
|\mathfrak{f}_t(x,u_1)-\mathfrak{f}_t(x,u_2)|
\leq L|u_1-u_2|,
\end{equation}
\begin{equation}
                                                                \label{lipschitz2a}
|f_t(x,u_1,z_1)-f_t(x,u_2,z_2)|
\leq
L(|u_1-u_2|+|z_1-z_2|)
\end{equation}
for all
$t\in[0,T]$, $x\in\cR$, $u_1, u_2\in\bR^d$ 
and  $z_1, z_2\in\mathbb R^{d\times d}$. 
\end{assumption}

\begin{assumption}
                                         \label{assumption3a}
We have $u_0\in \cH$. 
\end{assumption}
 
 Now we are in the position to present 
 our theorem on exponential stability, 
 where $\frc$ is the constant in the Poincar\'e inequality \eqref{Poincare}.  
\begin{theorem}
          \label{theorem4}
Let Assumptions \ref{assumption1a}, \ref{assumption2a} 
and \ref{assumption3a} hold. Then there is a constant $N=N(d,r,\kappa)$ 
such that if $u$ is an admissible solution to \eqref{NS3}--\eqref{bin3} 
satisfying 
$$
\tilde u+L<\delta/N,\quad \tilde u=\hat u+\esssup_{t\in[0,T]} \bar u_t,  
$$
then we have 
$$
\int_0^T|u_t-v_t|^2_{\cV}\leq \tfrac{1}{2\gamma}|u_0-v_0|^2_{\cH}
\quad\text{with} \quad \gamma:=\delta-N(\tilde u+L), 
$$
and 
$$
|u_t-v_t|_{\cH}\leq |u_0-v_0|_{\cH}e^{-\frc^{-2}\gamma t}
\quad 
\text{for $t\in[0,T]$}
$$
for any admissible solution $v$ to \eqref{NS} with zero boundary condition 
and initial condition $v_t|_{t=0}=v_0\in\cH$. 
\end{theorem}

Theorem \ref{theorem3} shows that any admissible solution to \eqref{SNS2}--\eqref{bin2}
is an  $\cH$-solution and it is unique in the restricted Leray-Hopf class 
whenever its Morrey component is sufficiently small. 
Hence Theorem \ref{theorem1} obviously follows by virtue of the 
following corollary.  
\begin{corollary}
                                                               \label{corollary1}
Let Assumptions 2.1, 2.2 and 2.3 hold, and let $u$ be a solution to
$(2.1)$--$(2.2)$ such that almost surely 
\begin{equation}
                                                        \label{LPS1}
\int_0^T\Big( \int_{\cR}|u_s(x)|^p\,dx\Big)^{q/p}\,ds < \infty
\end{equation}
for some
\begin{equation}
                                                                  \label{LPS2}
d<p\le\infty,
\qquad
2\le q<\infty,
\end{equation}
satisfying the Ladyzhenskaya--Prodi--Serrin condition
\begin{equation}
                                                                   \label{LPS3}
\frac{d}{p}+\frac{2}{q}\le1.
\end{equation}
Then $u$ is an $\mathcal H$-solution. Moreover, if $v\in L_{p',q'}$ is a solution 
to $(2.1)$--$(2.2)$, such that $p',q'$ satisfy \eqref{LPS2}-\eqref{LPS3},  
then for its $\mathcal H$-valued c\`adl\`ag modification,
still denoted by $v$, we have 
$$
u_t=v_t\quad\text{almost surely for all } t\in[0,T]. 
$$
\end{corollary}
\begin{proof}
By virtue of Theorem \ref{theorem3} we need only show that if 
$u$ is and $\bR^d$-valued  $\cO\otimes\cB(\cR)$-measurable 
function on $\Omega\times[0,T]\times\cR$ 
such that almost surely \eqref{LPS1} holds for a pair $(p,q)$ satisfying 
\eqref{LPS2}--\eqref{LPS3}, then $u$ is an admissible function 
such that $\hat u$ is as small as we wish.  
The argument proving this was given first in \cite{K2022b} and then 
used also in \cite{GK2026}. For the convenience of the reader 
we present it also here. 
If $p<\infty$ then let $r=d$ and set 
$$
\zeta_t=c\Big(\int_{\cR}
|u_t(x)|^{p }\,dx\Big)^{1/(p -d )} 
\quad\text{with a constant $c$}, 
$$
and define 
$u^{M}_t(x)=u_t(x){\bf1}_{|u_t(x)|\geq \zeta_t}$,
$u^B_t=u_t-u^M_t=u_t{\bf1}_{|u_t|\leq \zeta_t}$. 
Then for any $B\in\bB_{\rho}$ and $\varepsilon>0$ 
$$
\dashint_{B }|u^{M}_t(x)|^{d}\,dx
\leq \zeta_t^{d-p }
\dashint_{B }{\bf1}_{\cR}|u_t(x)|^{p }\,dx
\leq N(d)c^{d-p }\rho^{-d}\leq \varepsilon \rho^{-d}
$$
for sufficiently large $c=c(d,p,\varepsilon)$.  
Clearly, 
$|u_t^{B} |\leq \zeta_t $ such that  
$$
\int_{0}^{T}\zeta_t^{2}\,dt
=c^{2}\int_{0}^{T}|u_t|^{2p/(p-q) }_{L_p}
\,dt<\infty,\,\, \text{since $2p/(p-q)\leq q$.}
$$
 If $p=\infty$ then we take $u^M=0$, $u^B=u$, 
$\hat u=0$ and $\bar u_t:=|u_t|_{L_{\infty}}$,  and notice that 
$$
\int_0^T\bar u_t^2\,dt=\int_0^T|u_t|^2_{L_{\infty}}\,dt<\infty\,\,\text{(a.s.), 
\,\,since $q\geq 2$},
$$
which completes the proof of the corollary. 
\end{proof}
\begin{remark}
 If $u\in L_{d,\infty}$ (a.s.) 
 then the above calculations show that $u$ remains admissible 
in the critical case  $r=d$. However, in order to obtain the corollary, 
we need additional assumptions to guarantee that the Morrey 
norm of $u^M_t$ is sufficiently small, uniformly 
in $(\omega,t)$.  We note that in the three-dimensional case ($d=3$)  uniqueness 
and regularity results 
for Hopf-Leray weak solutions of the  
deterministic Navier-Stokes equations \eqref{eq1}--\eqref{bin1} 
(with $\frf=0$, $f=0$) were established in \cite{ESS2003} under the assumption that 
the solutions belong to $L_{d,\infty}$. More generally for all $d\geq3$ 
analogous uniqueness and regularity results were proved in \cite{DoDa2009}.  
\end{remark}
 
In the proofs of Theorems \ref{theorem2}, \ref{theorem3} and \ref{theorem4}
we will utilise a version of It\^o's formula for 
Banach space-valued semimartingales in the framework of Gelfand triplets
$
V \hookrightarrow H \equiv H^{\ast} \hookrightarrow V^{\ast},
$
where $H$ is a Hilbert space, $V$ is a Banach space, 
and $H^\ast$ and $V^\ast$ denote their dual spaces. 
In the next section we present this version which is an extension 
of the It\^o formula 
from in \cite{GK1982} in the direction when the drift component of 
of the semimartingale is a the sum of a $V^{\ast}$-valued process 
and an $H$-valued process such that $(|v_t^{\ast}|_{V^{\ast}})_{t\geq0}$ is locally 
square integrable, but $h=(|h_t|_H)_{t\geq 0}$ is only   
locally integrable in time $t$.
Although this extension may appear to be a minor technical generalization, 
it can in fact play an essential role in many applications, 
as demonstrated in \cite{GyongyKrylov} and \cite{GK2026}   
in the special case where both $H$ and $V$ are Hilbert spaces.

\mysection{ An It\^o's formula for semimartingales in Banach Spaces}

Let $(V,\|\cdot\|_{V})$ be a separable and reflexive real Banach 
space which is continuously and densely 
embedded in a Hilbert space $(H, |\cdot|)$. Thus we have 
$$
V\hookrightarrow H\equiv H^{\ast}\hookrightarrow V^{\ast}, 
$$
where $H^{\ast}\hookrightarrow V^{\ast}$ is the adjoint of the embedding 
$V\hookrightarrow H$, and $H^{\ast}$ is identified with $H$ with the help 
of the inner product in $H$, denoted by $(\cdot,\cdot)$. 
We use the notation 
$\langle\cdot, \cdot\rangle$ for the duality pairing between $V$ and $V^{\ast}$ 
and $|\cdot|_{V^{\ast}}$ denotes the norm in $V^{\ast}$. 
Note that if $v^{\ast}\in V^{\ast}$ for some $i$, then its restriction to $V$ belongs 
to $V^{\ast}$ note that $\langle v^{\ast},v\rangle=(h,v)$   
for all $v\in V$ when $v^{\ast}=h\in H$. 
Let $m=(m(t))_{t\geq0}$ be an $H$-valued locally square integrable martingale 
such that $\langle m\rangle$, the Doob-Meyer process of $|m|^2_H$ is continuous. 
 
Finally let $v=(v(t))_{t\geq0}$, $(h(t))_{t\geq0}$ and $(v^{\ast}(t))_{t\geq0}$  
be progressively measurable processes with values in $V$, $H$ and $V^{\ast}$  
respectively, such that for every $T>0$
$$
\int_0^T\|v(t)\|^2+|v^{\ast}(t)|^2_{V^{\ast}}+|h(t)|\,dt<\infty\quad\text{(a.s.)}.
$$
Then the following version of It\^o's formula plays an important role 
in this paper.  

\begin{theorem}
                                            \label{theorem0}

Suppose that for each $\varphi \in V$ 
\begin{equation}
                                                                                     \label{eq:1}
(v(t),\varphi)=
\int_{(0,t]}\langle v^{\ast}(s),\varphi\rangle\, dA(s) 
+\int_{(0,t]}(h(s),\varphi)\, dA(s)
+ (m(t),\varphi), 
\end{equation}
for $P\otimes dt$-almost every $(t,\omega)\in\Omega\times(0,\infty)$. 
Then there is a $\tilde\Omega\subseteq\Omega$ of full probability 
and
an $H$-valued c\`adl\`ag process $\tilde v$ such that 
the following statements hold: 
\begin{enumerate}
\item[(i)] For $dA\times \P$ almost all $(t, \omega)$ 
satisfying $t \in (0,\tau(\omega))$ we have $\tilde v=v$.
\item[(ii)] For all $\omega \in \tilde{\Omega}$ 
and $t\geq0$ we have
\begin{align}            
(\tilde v(t),\varphi)= &
 \int_{(0,t]}\langle v^{\ast}(s),\varphi\rangle\,ds                      \nonumber\\
 &+\int_0^t(h(s),\varphi)\,ds
 + (m(t),\varphi)\quad \textrm{for all }\, \varphi \in V.                \nonumber
\end{align}
\item[(iii)] For all $\omega \in \tilde{\Omega}$ and $t\geq0$ we have 
\begin{align}
|\tilde v(t)|^2  = 
& |m(0)|^2 + 2\sum_{i=1}^k\int_{(0,t]}\langle v^{\ast}_i(s),v(s)\rangle\,ds\nonumber\\ 
& +2 \int_{(0,t]}(h(s),v(s))\,ds+ 2\int_{(0,t]}( \tilde v(s-),dm(s))
+ [m](t). \nonumber 
\end{align}
\end{enumerate}
\end{theorem}

\smallskip
\noindent
This theorem can be obtained by an adaptation of the proof 
of Theorem 1 in \cite{GK1982}.

\mysection{Proof of the main results}

First we need some lemmas and corollaries. 
The first one is from \cite{K2022b}. 
\begin{lemma}
                                      \label{lemma k}
Let $f$ be a real-valued admissible function. Then
\begin{equation}
                                    \label{1.26.6}
\int_{\cR}|f_t^M|^2|\varphi|^2\,dx
\leq N\hat f^{\,2}\int_{\cR}|D\varphi|^2+|\varphi|^2\,dx
\,\, \text{$(\P\otimes dt$-a.e.}).
\end{equation}
 for $\varphi\in C_0^\infty(\cR,\bR)$, where
$N=N(d,r)$ is a constant depending on $d$ and $r$.
\end{lemma}
Define the trilinear form
\begin{equation*}
                                               \label{trilinear_form_definition}
b(u,v,w):=\int_{\bR^d}u^j(x)D_j v^i(x)w^i(x)\,dx,
\end{equation*}
for $\bR^d$-valued functions $u,v,w$ on $\bR^d$ 
whenever the integral is well defined. 
Then by \ref{lemma k} we have the following propositions. 
\begin{proposition}
                             \label{proposition1}
Let $w$ be an $\R^d$-valued admissible function and let
$v,w\in C_0^{\infty}(\cR,\bR^d)$. 
Then for $\P\otimes dt$-almost every $(\omega,t)$
we have 
\begin{align}
|b(u,v,w_t)|
&\leq N\hat w \big(|Du|_{L_2}+|u|_{L_2}\big)|Dv|_{L_2}  \nonumber\\
&\quad+N\bar w_t|Dv|_{L_2}|Dv|_{L_2}             \nonumber
\end{align}
with a constant $N=N(d,r)$. 
\end{proposition}

\begin{proof}
Clearly, 
$$
(u^{k}D_k v,w_t)= (w_t^M u^k,D_kv)+(w^B_t u^k,D_k v). 
$$
By using the Cauchy-Schwarz-Bunyakovsky inequality and then 
\eqref{1.26.6} with $f=w_t^{Mi}$ 
and $\varphi=u^k$ (for each $k,i\in\{1,2,...,d\}$), we get 
$$
(w^M_tu^k,D_kv)\leq N\hat w(|Du|_{L_2}
+|u|_{L_2})|Dv|_{L_2}
$$
with a constant $N=N(d,r)$. 
Hence we get the corollary by noting that 
$$
(w^B_tu^k,D_kv)\leq \int_{\cR} \big(\sum_{i,k}|w^{Bi}u_k|^2\big)^{1/2}
\big(\sum_{i,k}|D_kv^i|^2\big)^{1/2}\,dx\leq \bar w_t|u|_{L_2}|Dv|_{L_2}. 
$$
\end{proof}

\begin{proposition}
                             \label{proposition2}
Let $u$ be an $\R^d$-valued admissible function and let
$v,w\in C^{\infty}_0(\cR,\bR^d)$. 
Then for $\P\otimes dt$-almost every $(\omega,t)$ we have 
\begin{align}
|b(u_t,v,w)|&\leq N\hat u \big(|Dw|_{L_2}|w|_{L_2}\big)|Dv|_{L_2}  
                                                                                  \nonumber\\
&\quad+\bar u_t|Dv|_{L_2}|w|_{L_2}                         \nonumber
\end{align}
with a constant $N=N(d,r)$. 
\end{proposition}
\begin{proof}
The proof goes in the same way as that of the previous proposition.
\end{proof}

\begin{proposition}
                            \label{proposition antisymmetry}
Let $u$ be an $\mathbb R^d$-valued admissible function and let
$v,w\in \bH^1_0$. 
Then $|b(u_t,v,w)|<\infty$ \rm{($\P\otimes dt$-a.e.)}. 
Moreover, if $u_t\in\mathcal V$, 
then
\begin{equation*}
b(u_t,v,w)=-b(u_t,w,v), 
\end{equation*}
and hence $b(u_t,v,v)=0$. 
\end{proposition}
\begin{proof}
Clearly, $|b(u_t,v,w)|<\infty$ \rm{($\P\otimes dt$-a.e.)} 
by Proposition \ref{proposition2}. 
For $\varepsilon>0$ let 
$\varphi_{\varepsilon}\in C_0^{\infty}(\cR,\bR^d)$ such that 
$|v_t-\varphi_{\varepsilon}|_{\bH^1_0}\leq\varepsilon$. Then 
by integration by parts, by the Leibniz rule 
and using that ${\rm div}\, u_t=0$,  we get 
\begin{equation}
                                                                 \label{1.12.2}
\int_{\bR^d}u^j_tD_j\varphi^i_{\varepsilon}w^i\,dx
=-\int_{\bR^d}\varphi_{\varepsilon}^iD_j(w^iu_t^j)\,dx 
=-\int_{\bR^d}u^j_tD_jw^i\varphi_{\varepsilon}^i\,dx. 
\end{equation}
By Proposition \ref{proposition1} we have 
$$
|b(u_t,v-\varphi_{\varepsilon},w)|\leq N\hat u(|Dw|_{L_2}
+|w|_{L_2})|Dv-D\varphi_{\varepsilon}|_{L_2}
$$
$$ 
+N\bar u_t|Dv-D\varphi_{\varepsilon}|_{L_2}|w|_{L_2},  
$$
which shows that 
$$
b(u_t,\varphi_{\varepsilon},w)\to b(u_t,v,w)
\quad \text{for $\varepsilon\to0$}.
$$
In the same way we get 
$\lim_{\varepsilon\to0} b(u_t,w, \varphi_{\varepsilon})= b(u_t,w, v)$. 
Thus letting $\varepsilon\to0$ in \eqref{1.12.2} we finish the proof 
of the lemma. 
\end{proof}

\begin{proof}[Proof of Theorem \ref{theorem2}]
Without loss of generality we may assume that 
$$
\E|u_0|_{\mathcal H}^2
+\E\int_0^T|\frF_t|_{ L_2}^2\,dt
+\E\left(\int_0^T|\frG_t|_{L_2}\,dt\right)^2<\infty. 
$$
Let $u=u^M +u^B$ be an admissible solution on $[0,T]$, and 
set $H=\cH$, $V=\cV$. Then clearly, $V$ is a separable Hilbert space,  
which is continuously and densely embedded into the Hilbert 
space $H$. 

We define the processes
$$
v_t^\ast\in V^\ast,
\quad
F_t\in H,
\quad
G_t=(G_t^k)_{k\in\bZ_0}\in \ell_2(H),
$$
by requiring 
\begin{equation} 
                                            \label{v*_varation}
\langle v,v_t^\ast\rangle_{V,V^*}
= -\langle a_t^{ij}D_j u_t+\frf_t^i(u_t),D_i v\rangle_{V,V^*}
-\langle u_t^MD_j u_t,v\rangle_{V,V^*} ,
\end{equation}
\begin{equation*}
(v,F_t)_H
=
(-u_t^BD_j u_t+f_t(u_t,Du_t)+\gamma_t^{ik}D_i g_t^k,v),
\end{equation*}
and
$$
(v,G_t^k)_H=(\sigma_t^{ik}D_i u_t+h_t^k(u_t),v)_H,
\quad k\ge1
$$
for every $v\in V$. 
By Assumptions \ref{assumption1}, \ref{assumption2}(i) and 
Proposition \ref{proposition2},
$$
|\langle a_t^{ij}D_j u_t+\frf_t^i(u_t),D_i v\rangle _{V,V^*}|
\leq
(\delta^{-1}|Du_t|_{L_2}+\lambda_t|u_t|_{L_2}
+|\mathfrak{ F}_t|_{L_2})|v|_V,
$$
and
$$
|(u_t^MD_j u_t,v)|
\leq N\hat u|v|_{V}|Du_t|_{L_2}
$$
with a constant $N=N(d,r)$.  
Hence \eqref{v*_varation} defines a $V^{\ast}$-valued process $v^{\ast}$ 
such that 
$$
|v^{\ast}_t|_{V^{\ast}}
\leq (\delta^{-1}+N\hat u)|Du_t|_{L_2}+\lambda_t |u_t|_{L_2}
+|\mathfrak{ F}_t|_{L_2}.
$$
Thus taking into account that almost surely 
$$
\int_0^T \lambda_t^2|u_t|_{L_2}^2\,dt
\leq
\sup_{t\in[0,T]}|u_t|_{L_2}^2\int_0^T \lambda_t^2\,dt<\infty,
\quad \int_0^T|Du_t|^2_{L_2}\,dt<\infty
$$
and 
$$
\int_0^T|\frF_t|^2_{L_2}\,dt<\infty, 
$$
we get that 
\begin{equation}
                                                                  \label{Q_def}
\int_0^T |v_t^\ast|_{V^\ast}^2\,dt<\infty\,\,\text{(a.s.).}
\end{equation} 
Using the definition of admissible solution and Assumption 2.1(i),
we have 
$$
|-u_t^BD_j u_t+f_t(u_t,Du_t)|_{L_2}
\leq
\bar u_t|Du_t|_{L_2}+\kappa_t|u_t|_{L_2}+\chi_t|Du_t|_{L_2}
+|\mathfrak G_t|_{L_2},
$$
whence
\begin{equation*}
\int_0^T \kappa_t|u_t|_{L_2}\,dt
\le
\sup_{t\in[0,T]}|u_t|_{L_2}
\int_0^T \kappa_t\,dt<\infty\,\,\text{(a.s.)}.
\end{equation*}
Moreover, by Cauchy--Schwarz,
$$
\int_0^T(\bar u_t|Du_t|_{L_2}+\chi_t|Du_t|_{L_2})\,dt
<\infty \,\,\text{(a.s.)}.
$$
Using \eqref{closure_assumption}, Assumptions \ref{assumption1} and 
\ref{assumption2}(i),
$$
|\gamma_t^{ki}D_i q_t^k|_{L_2}
=
|\gamma_t^{ki}(RM_t^k(u_t))^i|_{L_2}
\le
\vartheta_t
\Bigl(
\sum_{k\in\bZ_0} |RM_t^k(u_t)|_{L_2}^2
\Bigr)^{1/2}
$$
$$
\leq
\vartheta_t
\Bigl(
\sum_{k\in\bZ_0} |M_t^k(u_t)|_{L_2}^2
\Bigr)^{1/2}
\leq
N\bigl(
\vartheta_t|Du_t|_{L_2}
+\vartheta_t\lambda_t|u_t|_{L_2}
+\vartheta_t|\mathfrak{F}_t|_{L_2}\big). 
$$
Hence taking account that 
\begin{equation*}
\int_0^T
\vartheta_t\lambda_t|u_t|_{L_2}\,dt
\le
\sup_{t\in[0,T]}|u_t|_{L_2}
\int_0^T(\vartheta_t^2+\lambda_t^2)\,dt
<\infty\,\,\text{(a.s.)},
\end{equation*}
and 
$$
\int_0^T\vartheta_t|Du_t|_{L_2}
+
\vartheta_t|\mathfrak{F}_t|_{L_2}\,dt
\leq 
\int_0^T\vartheta_t^2+|Du_t|^2_{L_2}+|\mathfrak{F}_t|^2_{L_2}\,dt
<\infty\,\text{(a.s.)}, 
$$
we get that 
\[
F_t
=
S\bigl(
-u_t^BD_j u_t
+f_t(u_t,Du_t)
+\gamma_t^{ki}D_i q_t^k
\bigr)
\]
is an $H$-valued optional process satisfying
\begin{equation}
                                                              \label{F}
\int_0^T |F_t|_H\,dt<\infty\,\,\text{(a.s.)}.
\end{equation}
Furthermore, by Assumptions \ref{assumption1}(i), \ref{assumption2}(i) 
and Remark \ref{remark sigma},
$$
\sum_k
|\sigma_t^{ik}D_i u_t+h_t^k(u_t)|_{L_2}^2
\leq
N\bigl(
|Du_t|_{L_2}^2+\lambda_t^2|u_t|_{L_2}^2
+|\mathfrak{F}_t|_{L_2}^2 
\big)
$$
with $N=N(\delta)$. Hence, 
$$
G_t=(G_t^k)_{k\in\bZ_0}=(S(\sigma_t^{ik}D_i u_t+h_t^k(u_t)))_{k\in\bZ_0}
$$
is an $\ell_2(H)$-valued optional process satisfying
\begin{equation}
                                            \label{G}
\int_0^T |G_t|_{\ell_2(H)}^2\,dt<\infty\,\,\text{(a.s.)}.
\end{equation}
Thus by virtue of \eqref{Q_def}, \eqref{F} and 
\eqref{G} we can apply Theorem \ref{eq:1} to get
\begin{equation}
                              \label{ito_sns}
|u_t|_H^2=|u_0|_H^2+\int_0^tI_s+(u_{s(u_s)},u_s)\,ds
+2 \int^{t}_{0} (u_{s-},d\frm_t)+[\frm]_t,
\end{equation}
where 
$
I_t:=-2\big(a_t^{ij}D_i u_t+2\frf_t^i(u_t),D_j u_t\big)
+2\big(f_t(u_t,Du_t)+2\gamma_t^{kj}D_j q_t^k,u_t\big), 
$
and 
$$
\frm_t=\int_0^tS\big(\sigma_s^{ik}D_i u_s+h_s^k(u_s)\big)\,dm_s^k. 
$$
By Proposition \ref{proposition antisymmetry},
$(u_{t(u_t)},u_t)=0$
($P\otimes dt$-a.e.). 
Unless otherwise stated, $N$ and $N'$ below 
denote constants depending only on $\delta$, 
and their values may change from one occurrence to another. 
Note that 
\begin{equation}
                                                 \label{DM}
<\frm>_t=\int_0^tJ_s\,ds\quad  
\text{with}
\quad J_t=\sum_k|S\big(\sigma_t^{ik}D_i u_t+h_t^k(u_t)\big)|_H^2. 
\end{equation}
Using Assumptions \ref{assumption1},  
\ref{assumption2}(i), taking into account 
\eqref{closure_assumption} and applying Young's inequality, 
we get
$$
I_t\leq
-2(a_t^{ij}D_i u_t,D_j u_t)
+
2\lambda_t|u_t|_{L_2}|Du_t|_{L_2}
+
2|\frF_t|_{L_2}|Du_t|_{L_2}
$$
$$
+
2\kappa_t|u_t|_{L_2}^2
+
2\chi_t|Du_t|_{L_2}|u_t|_{L_2}
+
2|\frG_t|_{L_2}|u_t|_{L_2}
+
2\vartheta_t|Du_t|_{L_2}|u_t|_{L_2} 
$$
\begin{equation}
                                              \label{I_estimate}
\leq-2(a_t^{ij}D_i u_t,D_j u_t)
+
\frac{\delta}{2}|Du_t|_{L_2}^2
+
N(\lambda_t^2+\chi_t^2+\vartheta_t^2+\kappa_t)
|u_t|_{L_2}^2
\end{equation}
$$
\quad
+
N|\frF_t|_{L_2}^2
+
2|\frG_t|_{L_2}|u_t|_{L_2}.   
$$
 Taking into account that the operator norm of $S$ is 1, 
expanding the square and applying Young's inequality we have 
\begin{equation}
                                                   \label{J_estimate}
J_t\leq(\sigma_t^{ik}\sigma_t^{jk}D_i u_t,D_j u_t)
+
\frac{\delta}{2}|Du_t|_{L_2}^2+N\lambda_t^2|u_t|_{L_2}^2
+
N|\mathfrak{ F}_t|_{L_2}^2. 
\end{equation}
Let
$$
Y_t:=e^{-2\varphi_t}|u_t|_{\mathcal H}^2,
\quad
\varphi_t:=\int_0^t\alpha_s\,ds 
\quad 
\tilde{\frm}_t:=\int^{t}_{0}e^{-2\varphi_s}(u_s,d\frm_s) .
$$
Then by It\^o's formula from \eqref{ito_sns} we get   
\begin{equation}
                                       \label{application_ito}
dY_t
=
e^{-2\varphi_t}
\bigl(I_t-2\alpha_t|u_t|_{\mathcal H}^2\bigr)\,dt
+
e^{-2\varphi_t}\,d[\frm]_t
+
2\,d\tilde{\frm}_t .
\end{equation}

From \eqref{I_estimate} and \eqref{J_estimate}, 
taking into account Assumption \ref{assumption1} we get 
$$
I_s+J_s\leq
-\delta |Du_s|_{L_2}^2
+
N_0(\lambda_s^2+\chi_s^2+\vartheta_s^2+\kappa_s)|u_s|_{\mathcal H}^2
$$
$$
+
N|\mathfrak{ F}_s|_{L_2}^2
+
2|\mathfrak{ G}_s|_{L_2}|u_s|_{\mathcal H}   
$$ 
with constants $N_0=N_0(\delta)$ and $N=N(\delta)$. Choosing
$$
\alpha_s=
N'(\lambda_s^2+\chi_s^2+\vartheta_s^2+\kappa_s)+\mu_s
$$
with $N'\geq N_0$, we get
$$
\E e^{-2\varphi_{T\wedge\tau}}
|u_{T\wedge\tau}|_{\mathcal H}^2
+
\delta\,\E\int_0^{T\wedge\tau}
e^{-2\varphi_s}|Du_s|_{L_2}^2\,ds
$$
$$
+\E\int_0^{T\wedge\tau}
\alpha_s e^{-2\varphi_s}|u_s|_{\mathcal H}^2\,ds
\leq
\E|u_0|_{\mathcal H}^2
+
N\E\int_0^{T\wedge\tau}
e^{-2\varphi_s}|\mathfrak {F}_s|_{L_2}^2\,ds
$$
\begin{equation}
                                                          \label{norm_estimate}
+2\E\int_0^{T\wedge\tau}
e^{-2\varphi_s}
|\mathfrak{G}_s|_{L_2}|u_s|_{\mathcal H}\,ds<\infty .
\end{equation}
Using this and estimate \eqref{I_estimate}, 
from \eqref{application_ito} we have 
\begin{equation}
                                    \label{sup_norm_estimate_1}
\begin{aligned}
E\sup_{t\le T\wedge\tau}Y_t
&\le
E|u_0|_{\mathcal H}^2
+
NE\int_0^{T\wedge\tau}
e^{-2\varphi_s}|\mathfrak {F}_s|_{L_2}^2\,ds  \\
&\quad
+
2E\int_0^{T\wedge\tau}
e^{-2\varphi_s}|\mathfrak {G}_s|_{L_2}|u_s|_{\mathcal H}\,ds  \\
&\quad
+
E\int_0^{T\wedge\tau}
e^{-2\varphi_s}\,d[\frm]_s
+
2E\sup_{t\le T\wedge\tau}
\left|\tilde\frm_s\right|.
\end{aligned}
\end{equation}
Using \eqref{J_estimate}, 
Remark \ref{remark sigma} and \eqref{norm_estimate} 
we get 
\begin{equation}\label{weighted_qv_estimate}
\begin{split}
\E\int_0^{T\wedge\tau}
e^{-2\varphi_s}\,d[\frm]_s
\leq&
N E|u_0|_{\mathcal H}^2
+
 N\E\int_0^{T\wedge\tau}
e^{-2\varphi_s}|\mathfrak {F}_s|_{L_2}^2\,ds
\\&+
N\E\int_0^{T\wedge\tau}
e^{-2\varphi_s}|\mathfrak {G}_s|_{L_2}|u_s|_{\mathcal H}\,ds.  
\end{split}
\end{equation}
By the  Davis inequality have  
$$
E\sup_{t\le T\wedge\tau}
\left|\tilde\frm_s\right|
\leq
3\E\left(
\int_0^{T\wedge\tau}
e^{-4\varphi_s}J_s|u_s|_{\mathcal H}^2\,ds
\right)^{1/2}. 
$$
Therefore
\begin{align}
\E\sup_{t\le T\wedge\tau}Y_t
&\leq
N\E|u_0|_{\mathcal H}^2
+
 N\E\int_0^{T\wedge\tau}
e^{-2\varphi_s}|\mathfrak {F}_s|_{L_2}^2\,ds  \nonumber\\
&\quad
+
 N\E\int_0^{T\wedge\tau}
e^{-2\varphi_s}|\mathfrak {G}_s|_{L_2}|u_s|_{\mathcal H}\,ds  \nonumber \\
&\quad
+
N_1\E\left(
\int_0^{T\wedge\tau}
e^{-4\varphi_s}J_s|u_s|_{\mathcal H}^2\,ds                               \nonumber
\right)^{1/2}<\infty. 
\end{align}
By Hölder and Young inequalities,
\begin{equation}
                                                                  \label{int_g_estimate}
\begin{aligned}
 E\int_0^{T\wedge\tau}
e^{-2\varphi_s}|\mathfrak {G}_s|_{L_2}|u_s|_{\mathcal H}\,ds
&\le
\frac14 E\sup_{t\le T\wedge\tau}Y_t  \\
&\
+
N \E\left(
\int_0^{T\wedge\tau}
e^{-\varphi_s}|\mathfrak {G}_s|_{L_2}\,ds
\right)^2<\infty .
\end{aligned}
\end{equation} Next, by Davis inequality,
$$
E\sup_{t\le T\wedge\tau}
\left|\tilde \frm_s\right|
\leq
3\E\left(
\int_0^{T\wedge\tau}
e^{-4\varphi_s}J_s|u_s|_{\mathcal H}^2\,ds
\right)^{1/2}.
$$
Hence 
$$
2\E\sup_{t\le T\wedge\tau}
\left|\tilde\frm_s\right|
\leq
\frac18\ E\sup_{t\le T\wedge\tau}Y_t
+
N \E\int_0^{T\wedge\tau}e^{-2\varphi_s}J_s\,ds<\infty .
$$
Using \eqref{J_estimate} together with \eqref{norm_estimate}, we get
\begin{equation}
                       \label{int_js_estimate}
\begin{aligned}
 \E\int_0^{T\wedge\tau}e^{-2\varphi_s}J_s\,ds
&\leq
N \E|u_0|_{\mathcal H}^2
+
N \E\int_0^{T\wedge\tau}
e^{-2\varphi_s}|\mathfrak {F}_s|_{L_2}^2\,ds
\\&+
\frac{1}{16}\E\sup_{t\le T\wedge\tau}Y_t
+N\E\left(
\int_0^{T\wedge\tau}
e^{-\varphi_s}|\mathfrak {G}_s|_{L_2}\,ds\right)^2.  
\end{aligned}
\end{equation}
Combining the estimates 
\eqref{sup_norm_estimate_1}--\eqref{int_js_estimate} we 
get 
$$
\E\sup_{t\le T\wedge\tau}Y_t
\leq
N \E|u_0|_{\mathcal H}^2
+
N \E\int_0^{T\wedge\tau}
e^{-2\varphi_s}|\mathfrak {F}_s|_{L_2}^2\,ds
$$
$$
\quad
+
N \E\left(
\int_0^{T\wedge\tau}
e^{-\varphi_s}|\mathfrak {G}_s|_{L_2}\,ds
\right)^2
+
\frac12\E\sup_{t\le T\wedge\tau}Y_t<\infty, 
$$
which yields  
$$
\E\sup_{t\le T\wedge\tau}
e^{-2\varphi_t}|u_t|_{\mathcal H}^2
\le
N\E|u_0|_{\mathcal H}^2
+
N \E\int_0^{T}
e^{-2\varphi_s}|\mathfrak {F}_s|_{L_2}^2\,ds
$$
\begin{equation}
                                 \label{final_sup_estimate}
+N \E\left(\int_0^{T}
e^{-\varphi_s}|\mathfrak {G}_s|_{L_2}\,ds
\right)^2.
\end{equation}
Using this estimate 
in \eqref{int_g_estimate}, from \eqref{norm_estimate} 
we obtain 
$$
\E Y_{T\wedge\tau}+\E\int_0^{T\wedge\tau}
e^{-2\varphi_s}|Du_s|_{L_2}^2\,ds
$$
$$
+
\E\int_0^{T\wedge\tau}\alpha_s e^{-2\varphi_s}
|u_s|_{\mathcal H}^2\,ds
\leq
N\E|u_0|_{\mathcal H}^2+N\E\int_0^{T}
e^{-2\varphi_s}
|\frF_s|_{L_2}^2\,ds
$$
\begin{equation}
                               \label{final_norm_estimate}
+
N\E
\left(
\int_0^{T}
e^{-\varphi_s}
|\mathfrak {G}_s|_{L_2}\,ds
\right)^2.
\end{equation}
We can finish now the proof of the theorem by using Fatou's lemma 
when taking $n\to\infty$ for $\tau:=\tau_n\uparrow\infty$ 
here and in \eqref{final_sup_estimate}. 
\end{proof}

\medskip
\begin{proof}[Proof of Theorem \ref{theorem3}]
Without loss of generality we may assume that 
$\E|u_0-v_0|_{H}^2<\infty$. Let $w:=u-v$. 
Then for each $\varphi\in \cD$ we have 
$$
(\fru_t,\varphi)
=-\int_0^t\big[(a^{ij}_sD_j\fru_s+\frf_s^{i}(u_s)-\frf^{i}_s(v_s),D_i\varphi)
+(u_{s(u_s)}-v_{s(v_s)}, \varphi)\big]\,ds
$$
$$
+\int_0^t
(f_s(u_s,Du_s)-f_s(v_s,Dv_s),\varphi)\,ds
$$
$$
+\int_0^t
\Big(\gamma_s^{ki}\big(\cR\big(M^k_s(u_s)-M^k_s(v_s)\big)\big)^i,\varphi\Big)
\,ds
$$
$$
+\int_0^t(\sigma_s^{ik}D_i\fru_s+h_s^k(u_s)-h_s^k(v_s),\varphi)\,dm^k_s
$$
for $P\otimes dt$-almost every $(\omega,t)\in [0,T]$, 
where $M_s^k(\cdot)$ is defined in \eqref{stochastic_integrand}. 
Setting $V:=\cV$, $H:=\cH$ and arguing 
as in the proof of Theorem \ref{theorem2}, 
we obtain an optional 
$V^*$-valued process $v^*$ such that
$$
\int_0^T |v_s^*|_{V^*}^2\,ds<\infty\,\text{(a.s.)}, 
$$
and for all $\varphi\in V$,
$$
\begin{aligned}
\langle \varphi,v_s^*\rangle
&=
-
\big(
a_s^{ij}D_j u_s
+
\frf_s^i(u_s)-\frf_s^i(v_s),
D_i\varphi
\big)
\\
&\quad
-
\big(
u_s^{Mi}D_i u_s-v_s^{Mi}D_i v_s,
\varphi
\big)
\end{aligned}
$$
for all $(\omega,t)$.
Moreover, define
$$
\begin{aligned}
F_s:={}&S\Big(v_s^{Bi}D_i v_s-u_s^{Bi}D_i u_s
+f_s(u_s,Du_s)-f_s(v_s,Dv_s)\Big)\\
&\quad
+
S\Big(\gamma_s^{ki}\big(R(M_s^k(u_s)-M_s^k(v_s))\big)^i\Big),
\quad t\in[0,T],
\end{aligned}
$$
and let 
$$
G_s=(G_s^k)_{k\in\bZ}
$$
be given by
$$
G_s^k
:=
S\big(
\sigma_s^{ik}D_i u_s
+
h_s^k(u_s)-h_s^k(v_s)
\big).
$$
Then $F_s$ is $\mathcal H$-valued and $G_s$ is
$\ell_2(\mathcal H)$-valued optional process, such that 
$$
\int_0^T|G_t|_{\ell_2(\mathcal H)}^2
+
|F_t|_{\mathcal H}\,dt<\infty\,\text{(a.s.).}
$$
Therefore we can use Theorem 3.1 to get that $w=(w_t)_{t\in[0,T]}$
admits an $\mathcal H$-valued c\`adl\`ag $P\otimes dt$-modification,
denoted also by $w$, such that almost surely
\begin{align}
|w_t|_{L_2}^2=&|w_0|^2_{\cH}-2\int_0^t
\big(a_s^{ij}D_j w_s+f_s^i(u_s)-f_s^i(v_s),D_j w_s\big)\,ds  \nonumber\\
&
-2\int_0^tb(u_s,u_s,w_s)-b(v_s,v_s,w_s)\,ds                         \nonumber\\
&+
2\int_0^t\big(f(u_s,Du_s)-f(v_s,Dv_s),w_s\big)\,ds                  \nonumber\\
&+2\int_0^t\Big(\gamma_s^{ki}
\big(R(M_s^k(u_s)-M_s^k(v_s))\big)^i,w_s\Big)\,ds                \nonumber\\
&+2\int_0^t(w_s,d\frm_s)+[\frm]_t                                         \label{1.12.6}
\end{align}
for all $t\in[0,T]$, where
$$
\frm_t=\int_0^tS\big(\sigma_s^{ik}D_i w_{s}+h_s^k(u_s)-h_s^k(v_s)\big)\,dm_s^k, 
\quad t\in[0,T], 
$$
and $[\frm]$ is the quadratic variation process of the locally square 
integrable martingale $\frm$.  
Notice that
$$
b(u_s,u_s,w_s)-b(v_s,v_s,w_s)=b(v_s,w_s,w_s)+b(w_s,u_s,w_s)
$$
\begin{equation}
                                                          \label{identity}
=b(w_s,u_s,w_s)=-b(w_s,w_s,u_s),   
\end{equation}
by Proposition \ref{proposition antisymmetry}. 
By Proposition  \ref{proposition1}, using Young's inequality 
we get 
$$
2|b(w_s,w_s,u_s)|\leq \hat N \hat u
\big(|Dw_s|_{L_2}+|w_s|_{L_2}\big)
|Dw_s|_{L_2}
$$
\begin{equation}
                                                            \label{1.16.6}
+2\bar u_s|Dw_s|_{L_2}|w_s|_{L_2} 
\end{equation}
$$
\leq \hat N\hat u|Dw_s|_{L_2}^2+\tfrac{\delta}{2}|Dw_s|_{L_2}^2
+
N_0(\hat u^{\,2}+\bar u_s^{\,2})|w_s|_{L_2}^2
$$
with constants $\hat N=\hat N(d,r)$ and $N_0=N_0(d,r,\delta)$.
Using Assumption \ref{assumption2}(ii) 
by the Cauchy-Schwartz
and Young inequalities we obtain, 
$$
\begin{aligned}
&
-2\big(a_s^{ij}D_j w_s+\frf_s^i(u_s)-\frf_s^i(v_s),D_i w_s\big)\\
&\leq
-2(a_s^{ij}D_j w_s,D_i w_s)+2\lambda_s|w_s|_{L_2}|Dw_s|_{L_2}\\
&\leq -2(a_s^{ij}D_j w_s,D_i w_s)+\tfrac{\delta}{4}|Dw_s|_{L_2}^2
+N\lambda_s^2|w_s|_{L_2}^2,
\end{aligned}
$$
with a constant $N=N(\delta)$.
Due to \eqref{lipschitz2}, we have 
$$
2\big(f(u_s,Du_s)-f(v_s,Dv_s),w_s\big)
$$
$$
\leq 2(\kappa_s|w_s|_{L_2}+\chi_s|Dw_s|_{L_2})|w_s|_{L_2}\\
\leq\tfrac{\delta}{8}|Dw_s|_{L_2}^2+N(\kappa_s+\chi_s^2)|w_s|_{L_2}^2
$$
with $N=N(\delta)$.
By condition \eqref{g_uperbound} we have  
\begin{align*}
K_s&:=\Big(
\gamma_s^{ki}\big(R(M_s^k(u_s)-M_s^k(v_s))\big)^i,u_s
\Big)                                                                           \\
&\leq\vartheta_s\left|\,
\big|R(M_s(u_s)-M_s(v_s))\big|_{\ell_2(\mathbb R^d)}\,|u_s|
\right|_{L_1}                                                              \\
&\leq\vartheta_s\left|\big|R(M_s(u_s)-M_s(v_s))\big|_{\ell_2(\mathbb R^d)}
\right|_{L_2}|u_s|_{L_2}.
\end{align*}
Since the operator norm of $R$ is $1$,
$$
\left|\big|R(M_s(u_s)-M_s(v_s))
\big|_{\ell_2(\mathbb R^d)}\right|_{L_2}^2
=\sum_{k=1}^{\infty}\left|R\big(M_s^k(u_s)-M_s^k(v_s)\big)\right|_{L_2}^2
$$
$$
\leq\sum_{k\in\bZ_0}\left|M_s^k(u_s)-M_s^k(v_s)\right|_{L_2}^2
=\big||M_s(u_s)-M_s(v_s)|_{\ell_2(\mathbb R^d)}\big|_{L_2}^2
$$
$$
\leq\left|\,\big|\sigma_s^i D_i w_s\big|_{\mathbb R^d}\big|_{\ell_2}
+\big|h_s(u_s)-h_s(v_s)\big|_{\ell_2(\mathbb R^d)}\right|_{L_2}^2 .
$$
Hence, 
by Remark \ref{remark sigma} and condition \eqref{lipschitz1},
$$
\begin{aligned}
|K_s|&\leq\vartheta_s
N\big(|Dw_s|_{L_2}+\lambda_s|w_s|_{L_2}\big)|w_s|_{L_2}\\
&\leq\tfrac{\delta}{16}|Dw_s|_{L_2}^2+
N'(\vartheta_s^2+\lambda_s^2)|w_s|_{L_2}^2,
\end{aligned}
$$
where $N,N'$ depend only on $d$ and $\delta$. 
Since the operator norm of $S$ is $1$, due to Assumption \ref{assumption1} and
condition \eqref{lipschitz1} we have 
$$
J_s:=\sum_k\Big|S(\sigma_s^{ik}D_i w_s+h_s^k(u_s)-h_s^k(v_s))
\Big|_{\mathcal H}^2
$$
$$
\leq\sum_k |\sigma_s^{ik}D_i w_s+h_s^k(u_s)-h_s^k(v_s)|_{\mathcal H}^2
$$
\begin{equation}
                                                            \label{[n]}
\leq (\sigma_s^{ik}\sigma_s^{jk}D_i w_s,D_j w_s)
+\tfrac{\delta}{32}|Dw_s|_{L_2}^2
+N_1\lambda_s^2|w_s|_{L_2}^2
\end{equation}
with a constant $N_1=N_1(d,\delta)$. Thus by the above estimates, from 
\eqref{1.12.6}
we obtain 
\begin{align}
d|w_t|^2_{\cH}
\leq& -2(a_s^{ij}D_j w_s,D_i w_s)+(\hat N\hat u+{\delta})|Dw_s|_{L_2}^2    \nonumber\\
&+N_1\beta|w_s|_{L_2}^2+2(w_t,d\frm_s)+d[\frm]_t                                     \label{1.13.6}
\end{align}
with a constant $N_2=N_2(d,r,\delta)$ and stochastic process 
$$
\beta_s=\hat u^2+\bar u_s^2+\lambda^2_s+\vartheta^2+\chi^2_s+\kappa_s. 
$$
Set 
$$
\phi_t=N'\int_0^t\beta_s\,ds' 
\quad\text{and}
\quad Y_t=e^{-\phi_t}|w_t|^2_{\cH}, 
\quad
 \tilde\frm_t=\int_0^te^{-\phi_s}(w_s,d\frm_s), 
$$
where $N'$ is a constant, specified later. Then from \eqref{1.13.6}
we get 
$$
dY_t\leq e^{-\phi_t}\big(-2(a_s^{ij}D_j w_t,D_i w_t)
+(\hat N\hat u+\tfrac{7}{8}{\delta})|Dw_t|_{L_2}^2\big)\,dt
$$
\begin{equation}
                                             \label{2.13.6}
+(N_2-N')\beta_t|w_t|^2_{\cH}\,dt +d [\tilde m]_t+2d\tilde m_t.                                                                                                                                                     
\end{equation}
Set
$$
\tau_n=\inf\{t\in[0,T], \eta_t\geq n\}\wedge\tilde\tau_n 
$$
\text{for integers $n\geq1$}, where 
$$
\eta_t=\int_0^t|Dw_s|^2_{L_2}
+(\bar u_s^2+\lambda^2_s+\vartheta^2+\chi^2_s+\kappa_s)|w_s|^2\,ds,  
$$
and $\tilde\tau_n\uparrow\infty$ is a localising 
sequence of stopping times for the local martingale 
$\tilde\frm$. Since $\eta$ is a continuous 
$\cF_t$-adapted process, $\tau_n$ is a stopping time for every $n$. 
Clearly, for $\tau:=\tau_n$ we have $\E\tilde\frm_{t\wedge\tau}=0$  
and due to \eqref{[n]} we have  
$$
\E[\tilde\frm]_{t\wedge\tau}=\E\langle\tilde\frm\rangle_{t\wedge\tau}
=\E\int_0^{t\wedge\tau}e^{-\phi_s}J_s\,ds
$$
$$
\leq \E\int_0^{t\wedge\tau}e^{-\phi_s}
\big((\sigma_s^{ik}\sigma_s^{jk}D_i w_s,D_j w_s)
+\tfrac{\delta}{32}|Dw_s|_{L_2}^2
+N_1\lambda_s^2|w_s|_{L_2}^2\big)\,ds<\infty
$$
for every $t\in[0,T]$. 
Hence, using that 
$$
-\big((2a_s^{ij}-\sigma_s^{ik}\sigma_s^{jk})D_j w_s,D_i w_s\big)
\leq
-2\delta|Dw_s|_{L_2}^2, 
$$
from \eqref{2.13.6} we get 
$$
\E Y_{t\wedge\tau}
\leq 
\E|w_0|_{\mathcal H}^2
$$
$$
+\E\int_0^{t\wedge\tau}
e^{-\phi_s}(\hat N\hat u-\delta)|Dw_s|_{L_2}^2
+(N_1+N_2-N')\beta|w_s|^2_{\cH}\,ds
$$
since $\E\tilde\frm_{\tau\wedge t}=0$. Hence, assuming that 
$\hat u\leq \delta/\hat u$ and choosing $N'\geq N_1+N_2$, 
we obtain 
$$
\E e^{-\varphi_{\tau_n\wedge t}}|w_t|^2_{\cH}
\leq \E |w_0|^2\quad\text{for $t\in[0,T]$} 
$$
for integers $n\geq1$. Letting here $n\to\infty$ and using Fatou's lemma 
we finish the proof of the theorem.  
\end{proof}

\medskip
\begin{proof}[Proof of Theorem \ref{theorem4}]
Setting $w=u-v$, we proceed as in the proof of Theorem \ref{theorem3} and apply 
Theorem \ref{theorem0} to obtain the `energy equality' 
$$
|w_t|_{L_2}^2=|w_0|_{\cH}^2-2\int_0^t
\big(a_s^{ij}D_j w_s, D_iw_s\big)+b(u_s,u_s,w_s)-b(v_s,v_s,w_s)\,ds
$$
$$
-2\int_0^t
(\frf^i_s(u_s)-\frf_s^i(v_s),D_iw_s)\,ds+2\int_0^t\big(f_s(u_s,Du_s)-f_s(v_s,Dv_s),w_s\big)\,ds.                                       
$$
From \eqref{identity} and \eqref{1.16.6}, using the Poincar\'e inequality \eqref{Poincare}
to control $|w|_{L_2}$ by $|Dw|_{L_2}$, we get
$$
2|b(u_s,u_s,w_s)-b(v_s,v_s,w_s)|
\leq N\tilde u|Dw_s|_{L_2}^2\,\,\text{($dt$-a.e. $s\in[0,T]$)} 
$$
with a constant $N=N(d,r,\frc)$ and $\tilde u=\hat u+\esssup_{s\in[0,T]}\bar u_s$.  
By conditions  \eqref{lipschitz1a} and \eqref{lipschitz2a} 
in Assumption \ref{assumption2a}, and using the Poincar\'e inequality,  
we obtain 
$$
2|(f_s^i(u_s)-f_s^i(v_s),D_i w_s)|
\leq 2\frc L|Dw_s|^2_{L_2}
$$
and 
$$
2|(f(u_s,Du_s)-f(v_s,Dv_s),w_s)|
\leq 2(\frc^2+\frc)L|Dw_s|^2_{L_2}. 
$$
Hence, using also the ellipticity condition from Assumption \ref{assumption1a},  
$$
(a_s^{ij}D_jw_s,D_iw_s)\geq \delta |Dw_s|^2\,\, \text{($ds$-a.e. $s\in[0,T]$)},   
$$
we obtain 
$$
|w_t|^2_{\cH}+\int_0^t|Dw_s|_{L_2}\leq |w_0|_{\cH}^2
+\int_0^2\big(2N(\tilde u+L)-2\delta\big)|Dw_s|^2\,ds,\quad \text{for $t\in[0,T]$}  
$$
with a constant $N=N(d,r,\frc)$. 
Consequently, if 
$$
\tilde u+L<\delta/N, 
$$ 
and we set 
$$
\gamma:=\delta -N(\tilde u+L)>0, 
$$
then 
$$
|w_t|^2_{\cH}\leq |w_0|_{\cH}^2
-\int_0^t 2\gamma |Dw_s|^2_{L_2}\,ds,   
\quad t\in[0,T].
$$
In particular, 
$$
\int_0^T |Dw_s|^2_{L_2}\,ds\leq \tfrac{1}{2\gamma}|w_0|^2,  
\quad t\in[0,T].  
$$
Furthermore, using the Poincar\'e inequality in the previous estimate yields 
$$
|w_t|^2_{\cH}\leq |w_0|_{\cH}^2
-2\gamma\frc^{-2}\int_0^t  |Dw_s|^2_{L_2}\,ds.  
\quad \text{for $t\in[0,T]$}, 
$$
and an application of Gronwall's lemma therefore gives 
$$
|w_t|_{\cH}\leq |w_0|_{\cH}e^{-\frc^{-2}\gamma t}, \quad t\in [0,T].    
$$
This completes the proof. 
\end{proof}

\mysection{Applications to Stochastic Navier-Stokes equations 
driven by wiener processes and 
Poisson martingale measures}

We now consider the following stochastic Navier-Stokes equations 
$$
du_t=\Big(D_i\big(a_t^{ij}D_j u_t + \frf_t^i(u_t))
+ f_t(u_t,Du_t)- u_{t(u_t)}-\nabla p_t
+ \gamma_t^{ki}D_i q_t^k\Big)\,dt
$$
\begin{equation}
                                                                \label{SNS5}
+\Big(\sigma_t^{ik}D_i u_t+ h_t^k(u_t)-\nabla q_t^k\Big)\,dw_t^k
+\int_{Z}\frh_t(u_{t-},z)\,\tilde\pi(dz,dt), 
\end{equation}
together with the incompressibility condition 
 \begin{equation}
                                                   \label{div5}
\operatorname{div}u_t=0
\end{equation}
on $\Omega\times[0,T]\times\cR$, subject to the 
the boundary and initial conditions 
\begin{equation}
                                                         \label{bin5}
u_t|_{\partial \cR}=0, \quad \text{for $t\in[0,T]$\quad 
and
\quad 
$u_t(x)|_{t=0}=u_0(x)$,
\quad$ x\in \cR$}.   
\end{equation}
Here  the coefficients $a=(a^{ij})$, $\gamma=(\gamma^{ki})$, 
$\sigma=(\sigma^{ik})$ and the functions $\frf=(\frf^i)$, $f$ are the same 
as those appearing in \eqref{SNS2}. 
The only difference is that, in \eqref{SNS5}, 
the summation over the repeated index $k$ is understood to run only over $k\geq1$.
The function $\frh=\frh_t(x,u,z)$ is an $\bR^d$-valued 
$\cO\otimes\cB(\cR\times\bR^d)\otimes\cZ$-measurable function 
of $\omega\in\Omega$, $t\geq0$, $x\in\cR$, $u\in\bR^d$ and $z\in Z$, 
satisfying the following assumption.   

\begin{assumption} 
                                                   \label{assumption2c}
There exist a nonnegative optional process $\tilde \lambda=(\tilde\lambda_t)_{t\geq0}$
and a nonnegative $\cO\otimes\cB(\cR)$-measurable function
$\frH=\frH_t(x)$ such that
$$
\int_0^T\tilde\lambda_t^2\,dt<\infty,
\quad
\int_0^T|\frH_t|_{L_2}^2\,dt<\infty
\quad \text{(a.s.)},
$$
and the following conditions hold: 
\begin{enumerate}
\item[(i)]
For all
$\omega\in\Omega$, $t\geq0$, $x\in\cR$ and $u\in\mathbb R^d$ we have  
$$
\int_{Z}|\frh_t(x,u,z)|^2\,\frn(dz)\leq \tilde\lambda^2_t|u|^2+\frH^2_t(x). 
$$
\item[(ii)]
For all $\omega\in\Omega$, $t\geq0$ we have  
$x\in\cR$, $u_1, u_2\in\bR^d$,  
$$
\int_{Z}|\frh_t(x,u_1,z)-\frh_t(x,u_2,z)|^2\,\frn(dz)
\leq \tilde\lambda^2_t |u_1-u_2|^2.
$$
\end{enumerate}
\end{assumption}

The notion of solution is defined analogously to that introduced above. 
\begin{definition}
                                                            \label{definition2}
A $\cV$-valued optional process
$u=(u_t)_{t\in[0,T]}$
is called a solution of \eqref{SNS5}--\eqref{div5}--\eqref{bin5} on $[0,T]$ if
$u\in L_2([0,T],\cV)\cap L_\infty([0,T],\cH)$ (a.s.),  
and for every $\varphi\in\cV$, 
\begin{align}
(u_t,\varphi)=&(u_0,\varphi)-
\int_0^t\Big[\big(a_s^{ij}D_j u_s+\frf_s^i(u_s),D_i \varphi\big)
+(u_{s(u_s)},\varphi)\Big]\,ds                                                                   \nonumber\\
&+\int_0^t\big(\gamma_s^{ki}D_i q_s^k+f_s(u_s,Du_s),\varphi\big)\,ds  \nonumber\\
&+\int_0^t\big(\sigma_s^{ik}D_i u_s+h_s^k(u_s),\varphi\big)\,dw_s^k    \nonumber\\
&+\int_0^{t}\int_{Z}\big(\frh_s(u_{s-},z),\varphi\big)\,\tilde\pi(ds,dz)                       \label{1.19.6}
\end{align}
holds for $\mathbb P\otimes dt$-almost every
$(\omega,t)\in\Omega\times[0,T]$. In addition 
equation \eqref{pressure_assumption} is assumed to hold. 

As before, we call
$u=(u_t)_{t\in[0,T]}$
an $\cH$-solution if it is an $\cH$-valued 
$\cF_t$-adapted c\`adl\`ag process such that
$u\in L_2([0,T],\cV)\ \text{(a.s.)}$, 
almost surely \eqref{1.19.6} holds for all
$t\in[0,T],\,\varphi\in\mathcal V$, 
and equation \eqref{pressure_assumption} is satisfied. 
Recall that a solution is called an admissible solution if it is 
an admissible function. 
\end{definition} 

Theorems \ref{theorem2} and \ref{theorem3} yield 
the following result. 

\begin{theorem}
                                            \label{theorem5}
Let Assumptions \ref{assumption1}, 
\ref{assumption2}, \ref{assumption3}  and \ref{assumption2c} hold.  
Then every admissible solution admits an $\cH$-valued c\`adl\`ag modification. 
Moreover,  there exists a constant 
$N=N(d,r)>0$
such that if $u$ is an admissible solution 
to \eqref{SNS5}--\eqref{div5}--\eqref{bin5} satisfying 
$\hat u\leq \delta/N$, then 
$$
\E e^{-\psi_{t}}|u_t-v_t|^2_{\cH}
\leq \E |u_0-v_0|_{\cH}^2\quad\text{for $t\in[0,T]$} 
$$
for any admissible solution $v$ of \eqref{SNS5}--\eqref{div5}--\eqref{bin5} 
on $[0,T]$ with $\cH$-valued initial condition 
$v_t|_{t=0}=v_0$, where 
$$
\psi_t=N\int_0^t\beta_s\,ds,
\quad 
\beta_s
=\hat u^2+\bar u_s^2+\lambda^2_s+\tilde\lambda^2_s+\vartheta^2+\chi^2_s+\kappa_s
$$
with a constant  $N=N(d,r,\delta)$. In particular, if $u_0=v_0$ 
then the $\cH$-valued c\`adl\`ag modifications 
of $u$ and $v$ coincide almost surely on  $[0,T]$. 
\end{theorem}
\begin{proof}
We adapt the method of rewriting stochastic integrals 
with respect to martingale measures 
as stochastic integrals with respect to martingales with values 
in Hilbert spaces, see \cite{KG1980}. 
Let $(e_k)_{k\geq 1}$ be an orthonormal basis of the Hilbert space 
$L_2((Z,\cZ,\frn),\bR^d)$, the space of $\bR^d$-valued $\cZ$-measurable functions on $Z$ that are square-integrable with respect to $\frn(dz)$.
Define the sequence $(m^k)_{k\in\bZ}$ of real-valued $\cF_t$-martingales 
by 
$$
m^k_t=w^k_t\quad\text{for $k\geq1$,
\qquad  $m^k_t=\int_0^t\int_{Z}e_k(z)\,\tilde\pi(dz,ds)$\quad for $k\leq0$}. 
$$
Then the predictable quadratic variation processes satisfy 
$$
\langle m^k,m^l\rangle_t=\delta_{kl}t\quad\text{for $k,l\in\bZ$}. 
$$  
Next define for every $k\in\bZ$ the 
$\bR^d$-valued function $\tilde h^k=\tilde h^k_t(x,u)$ by setting 
$$
\tilde h^k_t(x,u)=h^k_t(x,u)\quad\text{for $k\geq1$}, 
$$
and, for $k\leq0$,  
$$
\tilde h^k_t(x,u)= \int_{Z} \frh_t(x, u, z)e_k(z)\,\frn(dz),  
$$
for $\omega\in\Omega$, $t\geq0$, $x\in\cR$ and $u\in\bR^d$. 
Then for every $\varphi\in\cH$ we have 
$$
\int_0^{t}\int_{Z}\big(\frh_s(u_s,z),\varphi\big)\,\tilde\pi(ds,dz)
=\sum_{k\leq0}\int_0^{t}\int_{Z}\big(\tilde h^k_s(u_s),\varphi\big)\,dm^{k}_s  
$$
and 
$$ 
\int_0^t\big(h^k_s(u_s),\varphi\big)\,dw^k_s
+\int_0^{t}\int_{Z}\big(\frh_s(u_s,z),\varphi\big)\,\tilde\pi(ds,dz)
$$
$$
=\sum_{k\in\bZ}\int_0^{t}\big(\tilde h^k_s(u_s),\varphi\big)\,dm^{k}_s
$$
 almost surely for all $t\in[0,T]$.  Moreover, 
 define $\tilde\sigma^{ik}=\sigma^{ik}$ 
 for $k\geq1$ and $\tilde\sigma^{ik}=0$ for $k\leq0$.  
 Then, for every $\varphi\in\cV$,  we have that  
 $$
\int_0^t\big(\sigma^{ik}_s D_iu_s+h^k_s(u_s),\varphi\big)\,dw^k_s
+\int_0^{t}\int_{Z}\big(\frh_s(u_s,z),\varphi\big)\,\tilde\pi(ds,dz)
$$
$$
=\sum_{k\in\bZ}
\int_0^{t}\big(\tilde\sigma^{ik}_sD_iu_s+\tilde h^k_s(u_s),\varphi\big)\,dm^{k}_s
$$
almost surely for all $t\in[0,T]$. Consequently, the stochastic Navier-Stokes equations 
\eqref{SNS5} can be rewritten in the form of \eqref{SNS2}, where 
the coefficient $\sigma$ and the function $h$ in \eqref{SNS2} are 
replaced by $\tilde\sigma$ and $\tilde h$, and the driving sequence of martingales $(m^k)_{k\in\bZ}$ is given as above. 
It remains to verify that $\tilde\sigma$ and $\tilde h$ satisfy the assumptions in Assumptions \ref{assumption1} and \ref{assumption2}, respectively, which were imposed on $\sigma$ and $h$.

It is straightforward to verify that the coercivity condition \eqref{coercivity} remains valid when $\sigma$ is replaced by $\tilde\sigma$. 
By Parseval's identity and Assumption \ref{assumption2c}, 
we have 
$$
 \sum_{k\leq0}|\tilde h^k_t(x,u)|_{\bR^d}^2=\int_{Z}|\frh_t(x,u,z)|^2_{\bR^d}\,\frn(dz)
 \leq \tilde\lambda_t^2|u|^2+|\frH_t(x)|^2
$$
for all $\omega$, $t\in[0,T]$, $x\in\cR$, and $u\in\bR^d$. Moreover, 
$$
\sum_{k\leq0}|\tilde h^k_s(x,u_1)-\tilde h^k_s(x,u_2)|^2_{\bR^d}
=\int_{Z}|\frh_s(x,u_1,z)-\frh_s(x,u_2,z)|^2\,\frn(dz)
$$
$$
\leq \tilde\lambda_t^2|u_1-u_2|^2 
$$
for all $\omega$, $t\in[0,T]$ $x\in\cR$, and $u_1, u_2\in\bR^d$.  
Hence, using \eqref{f_upperbound} and \eqref{lipschitz1} 
from  Assumption \ref{assumption2},  
we obtain 
$$
|\tilde h_t(x,u)|_{\ell_2(\bR^d)}
\leq  (\tilde\lambda_t+\lambda_t)|u|+\frH_t(x)+\frF_t(x)
$$
and 
$$
|\tilde h_t(x,u_1)-\tilde h_t(x,u_2)|_{\ell_2(\bR^d)}
\leq  (\tilde\lambda_t+\lambda_t)|u_1-u_2|. 
$$
This shows that, with $\tilde h$ in place of $h$,  
Assumption \ref{assumption2} is satisfied with 
$\tilde\lambda+\lambda$ replacing $\lambda$, and 
$\frH+\frG$ replacing $\frG$. This completes the proof 
of the theorem.  
\end{proof}

\end{document}